\documentclass[a4paper,11pt]{article}
\usepackage{times}

\usepackage{amssymb,amsfonts,amsthm,amsmath}  %% <- after lineno
\usepackage[mathlines]{lineno}
\usepackage{hyperref}
\usepackage[bbgreekl]{mathbbol}
\usepackage{tikz}
\usetikzlibrary{decorations.pathmorphing,patterns}
\usepackage{stmaryrd}
\usepackage{authblk}
\usepackage{thmtools, thm-restate}
\usepackage[toc]{appendix}
\usepackage{draftwatermark}
\SetWatermarkLightness{0.95}
\SetWatermarkScale{1.0}
\SetWatermarkText{\color{red!10!white}{\rm DRAFT}}

\hypersetup{
    unicode=true,          % non-Latin characters in Acrobat’s bookmarks
    pdftoolbar=true,        % show Acrobat’s toolbar?
    pdfmenubar=true,        % show Acrobat’s menu?
    pdffitwindow=false,     % window fit to page when opened
    pdfstartview={FitH},    % fits the width of the page to the window
    pdftitle={My title},    % title
    pdfauthor={Author},     % author
    pdfsubject={Subject},   % subject of the document
    pdfcreator={Creator},   % creator of the document
    pdfproducer={Producer}, % producer of the document
    pdfkeywords={keyword1, key2, key3}, % list of keywords
    pdfnewwindow=true,      % links in new PDF window
    colorlinks=true,       % false: boxed links; true: colored links
    linkcolor=red!85!black,          % color of internal links (change box color with linkbordercolor)
    citecolor=red!85!black,        % color of links to bibliography
    filecolor=cyan,         % color of file links
    urlcolor=red!85!black        % color of external links
}

\DeclareSymbolFontAlphabet{\mathbbl}{bbold}
\allowdisplaybreaks[1]

 \usepackage{color}
	\newcommand{\AM}{\textcolor{black}}
\newtheorem{prop}{Proposition}[section]
\newtheorem{remn}{Remark}[section]
\numberwithin{equation}{section}

\newcommand{\R}{{\sf I\!R}}
\newcommand{\Z}{\mathbb{Z}}
\newcommand{\weak}{\rightharpoonup}

\newcommand{\norm}[1]{\left \| #1 \right \|}

\newcommand{\jump}[1]{\left \llbracket #1 \right \rrbracket} 
\newcommand{\vc}[1]{{\boldsymbol #1}}
\newcommand{\mean}[1]{\left \langle#1 \right \rangle} 
\newcommand{\ep}{\varepsilon}
\newcommand{\one}{\raisebox{0.05pt}{$\cdot$}}
\newcommand{\two}{\raisebox{0.5pt}{$:$}}
\newcommand{\three}{\raisebox{-1.5pt}{$\vdots$}}

\newcommand{\thmref}[1]{\makebox{Theorem~\ref{#1}}}

\newcommand{\secref}[1]{\makebox{Section~\ref{#1}}}

\newcommand{\propref}[1]{\makebox{Proposition~\ref{#1}}}
\newcommand{\figref}[1]{\makebox{Fig.~{\rm\ref{#1}}}}
\pgfkeys{/tikz/.cd,
    rope width/.store in=\RopeWidth,
    rope width=5pt,
    rope step/.store in=\RopeStep,
    rope step=2mm,
}

\pgfdeclaredecoration{rope}{initial}
{% 
\state{initial}[width=\RopeStep,next state=cont] {
    \pgfmoveto{\pgfpoint{0pt}{-\RopeWidth/2}}
    \pgfpathcurveto
    {\pgfpoint{5*\RopeStep/6}{0.25*\RopeWidth}}
    {\pgfpoint{7*\RopeStep/6}{0.45*\RopeWidth}}
    {\pgfpoint{1.5*\RopeStep}{\RopeWidth/2}}
     \pgfpathcurveto
    {\pgfpoint{10*\RopeStep/6}{0.55*\RopeWidth}}
    {\pgfpoint{11*\RopeStep/6}{0.6*\RopeWidth}}
    {\pgfpoint{13.5*\RopeStep/6}{\RopeWidth/2}}
    \pgfcoordinate{lastup}{\pgfpoint{-1.5*\RopeStep/6}{-\RopeWidth/2}}
  }
  \state{cont}[width=\RopeStep]{ 
     \pgfmoveto{\pgfpointanchor{lastup}{center}}
     \pgfpathcurveto
    {\pgfpoint{-5*\RopeStep/6}{-0.6*\RopeWidth}}
    {\pgfpoint{-4*\RopeStep/6}{-0.55*\RopeWidth}}
    {\pgfpoint{-3*\RopeStep/6}{-0.55*\RopeWidth}}
     \pgfpathcurveto
    {\pgfpoint{-\RopeStep/6}{-0.45*\RopeWidth}}
    {\pgfpoint{\RopeStep/6}{-0.25*\RopeWidth}}
    {\pgfpoint{3*\RopeStep/6}{0pt}}
    \pgfpathcurveto
    {\pgfpoint{5*\RopeStep/6}{0.25*\RopeWidth}}
    {\pgfpoint{7*\RopeStep/6}{0.45*\RopeWidth}}
    {\pgfpoint{9*\RopeStep/6}{\RopeWidth/2}}
     \pgfpathcurveto
    {\pgfpoint{10*\RopeStep/6}{0.55*\RopeWidth}}
    {\pgfpoint{11*\RopeStep/6}{0.6*\RopeWidth}}
    {\pgfpoint{13.5*\RopeStep/6}{\RopeWidth/2}}
    \pgfcoordinate{lastup}{\pgfpoint{-1.5*\RopeStep/6}{-\RopeWidth/2}}
  }
  \state{final}[width=5pt]
  {
     \pgfmoveto{\pgfpointanchor{lastup}{center}}
     \pgfpathcurveto
    {\pgfpoint{-5*\RopeStep/6}{-0.6*\RopeWidth}}
    {\pgfpoint{-4*\RopeStep/6}{-0.55*\RopeWidth}}
    {\pgfpoint{-0.5*\RopeStep}{-0.55*\RopeWidth}}
     \pgfpathcurveto
    {\pgfpoint{-\RopeStep/6}{-0.45*\RopeWidth}}
    {\pgfpoint{\RopeStep/6}{-0.25*\RopeWidth}}
    {\pgfpoint{0.5*\RopeStep}{0pt}}
    \pgfmoveto{\pgfpointdecoratedpathlast}
  }
}

\title{Effective medium theory for\\ second-gradient elasticity with chirality}
\author{Grigor Nika\footnote{Corresponding author: \url{grigor.nika@kau.se} } \hspace{0.5ex} and~Adrian Muntean} 
\affil{Dept. Mathematics \& Computer Science\\ Karlstad University \\
Universitetsgatan 2\\ 651 88 Karlstad, Sweden \protect\\ \url{grigor.nika@kau.se}\protect\\ \url{adrian.muntean@kau.se}}

\begin{document}

\maketitle

\begin{abstract}
We derive effective models from a heterogeneous second-gradient nonlinear elastic material taking into account chiral scale-size effects. Our classification of the effective equations depends on the hierarchy of four characteristic lengths: The size of the heterogeneities $\ell$, the intrinsic lengths of the constituents $\ell_{\rm SG}$ and $\ell_{\rm chiral}$, and the overall characteristic length of the domain ${\rm L}$. Depending on the different scale interactions between $\ell_{\rm SG}$, $\ell_{\rm chiral}$, $\ell$, and ${\rm L}$ we obtain either an effective Cauchy continuum or an effective second-gradient continuum. The working technique combines scaling arguments with the periodic homogenization asymptotic procedure. Both the passage to the homogenization limit and the unveiling of the correctors' structure rely on a suitable use of the periodic unfolding and related operators.
\end{abstract}

\textbf{MSC 2020: }74Q05, 74B20, \AM{35B27}, 35G35, 35G45, 35J58, 35Q74 

\textbf{Keywords: }second-gradient elasticity, scale-size effects, partial scale separation, chirality

%\newpage
%\tableofcontents
%\newpage

\section{Introduction}

Contemporary advancements and developments in additive manufacturing technology have led to a widespread adoption of materials with microstructure. Typical engineered materials with microstructure include ceramic matrix composites, fibre-reinforced polymers, and many other advanced functional materials. What these aforementioned materials have in common, from the point of view of applications, is their properties. Macroscopically, materials with a hierarchical microstructure may have vastly different characteristic properties than those of the underlying microstructure. Hence, by exploiting sophisticated microstructures we can design and produce, programmable macroscopic material behavior, e.g., low weight to strength ratio of panels, desired buckling modes of beams, programmable negative Poisson's ratio materials, etc..; see for instance the examples reported in \cite{Bilal17}, \cite{Schum15}, \cite{ANC22}, \cite{AnCoNi20}. 

Generalized continuum theories (compare, e.g., \cite{Toupin62}, \cite{MT62}, \cite{Toupin64}, \cite{Mindlin64}, \cite{ES64I}, \cite{ES64II}, \cite{Mindlin65}, \cite{Eringen66}, \cite{ME68}, \cite{Now72}) have been consistently applied to modelling of materials with microstructure, such as granular or fibrous materials, or materials with a lattice structure (as well as other non-simple material, see, e.g., \cite{mielke2020}). Generalized continuum theories are largely split into higher-gradient methods (e.g., second-gradient material~\cite{MT62},~\cite{Toupin64},~\cite{Mindlin64},\cite{Mindlin65},~\cite{ME68}, \cite{Duv70}) or higher order methods (e.g., Cosserat material~\cite{CC1909}~\cite{Lakes83}, \cite{Lakes86}, \cite{Lakes93}, \cite{Lakes95}, \cite{Eringen99}, \cite{RL17}). Both theories are general enough to quantitatively delineate higher-gradients or higher-order models that incorporate chirality and microstructural scale-size effects. Scale-size effects refer to the changes in behavior or characteristics of a structure as its size is altered. Plainly put, it means that things can behave differently or have different properties depending on their size (see \figref{fig:size_effects}). Chiral (or non-centrosymmetric) materials, on the other hand lack a center of symmetry; they are not invariant under inversion of coordinates transformation (see \cite{HPL16}). Chirality may be present at different scales in the material and is a characteristic of engineered materials containing twisted fibres, e.g., wire rope, cables and even biological filaments, e.g., DNA strands (see \cite{Healey02}). 

\begin{figure}[!tbh]
    \centering
    \includegraphics[scale=0.35]{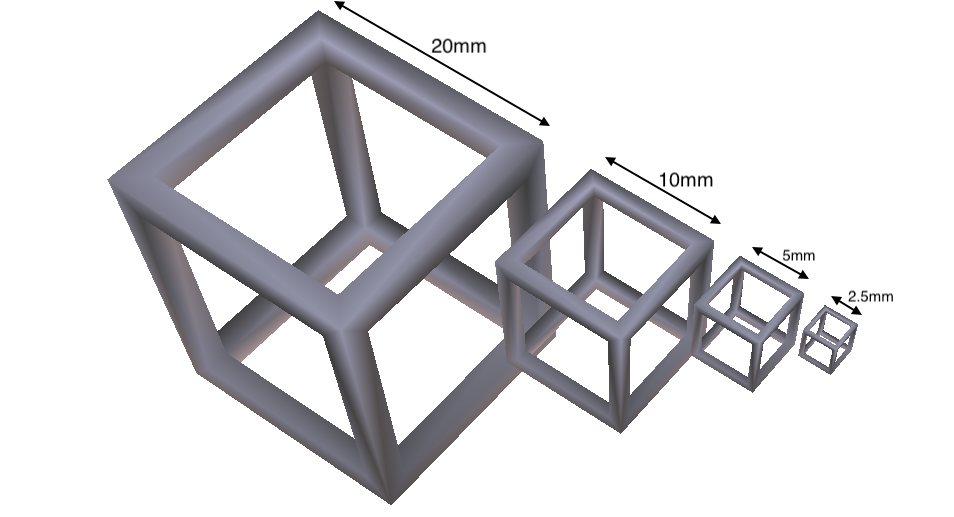}
    \caption{\small Scale-size effects highlight that the size or scale of a structure can influence its behavior, strength, as well as other properties. Within the domain of theoretical mechanics, conventional periodic homogenization theories rooted in the Cauchy continuum framework maintain their validity on the condition of pronounced scale separation. However, when the sizes of micro- and macro-structures converge, breaching the realm of comparability, these theories falter and succumb to the manifestation of size effects.}
    \label{fig:size_effects}
\end{figure}

Homogenization methods are particularly well suited for the analysis of heterogeneous materials with periodically distributed microstructures; for technical  details, we refer the reader for instance to ~\cite{Bens78}, \cite{SP80}, \cite{BP89}, \cite{All02book}, \cite{CioDo00}, \cite{MV10}. The technique of homogenization has been applied widely to derive effective equations, both of local and non-local nature, in mechanics, physics, chemistry, and in other natural sciences (see, e.g., \cite{bytner1988homogenization}, \cite{nika2023hyper}, \cite{Adrian_QAM}) since it can account for the influence of volume fraction, distribution, and morphology. Nevertheless, it is worth noting that the majority of models amenable to homogenization techniques adhere to the classical Cauchy material framework, which regrettably cannot capture scale-size effects due to the inherent size-independence of the classical elastic tensor. Furthermore, the aspect of chirality, a critical characteristic in certain materials, also remains unaddressed by classical Cauchy material. To circumvent this impasse, we propose a solution that entails the application of homogenization methods within an enriched continuum, thereby facilitating the incorporation of scale-size effects and the modeling of chirality. There are two potent ways of enriching the continuum: Allow higher gradients of the displacement field~\cite{TB96}, \cite{ZZA97}, \cite{SC2000} or allow additional degrees of freedom~\cite{FS98}, \cite{Forest01}, \cite{Forest13}. The current work focuses on the periodic homogenization within the confines of a linear approximation for a second-gradient nonlinear elastic material. The model proposed is sufficiently rich to model chiral-type microstructures and account for scale-size effects by means of dimensional analysis. In doing so, we rigorously derive two different classes of effective models: If the size of the heterogeneities is comparable with the period, then we obtain an effective classical Cauchy continuum. If the size of the heterogeneities is comparable with the overall length of the domain (no scale separation), then we obtain an effective second-gradient material. In the latter case, we recover the boundary conditions and the equilibrium equations for second-gradient theory as originally proposed in \cite{ME68}, \cite{Germain73}. Additionally, compared to the classical works in \cite{ME68}, \cite{Germain73}, we can now compute  explicitly the effective coefficients that characterize the material properties, taking into account volume fraction, particle distribution, and morphology. This is a novelty from the methodological point of view. Moreover, since we will be dealing with higher gradients, the choice of method to rigorously pass to the limit plays an important role. Certain techniques of homogenization lend themselves to be more easily exploited in dealing with higher-gradients than others. In this work we will use the method of periodic unfolding~\cite{CDG02, Dam05, CDG08, CDG18} instead of two-scale convergence \cite{Ngu89}, \cite{All92}, \cite{LNW02}. The unfolding method has a natural way of handling higher-gradients without any extra effort, as it was pointed out in the original work~\cite{CDG08}. Furthermore, the results presented here can be extended to domains with holes by adjusting the periodic unfolding operator as in \cite{CDDGZ12}.

To fix ideas, we designate an origin and the natural orthonormal basis in $\R^3$ and we choose the reference configuration to coincide with the natural or stress-free configuration. We denote by $\overline{\Omega}$ the region occupied in the reference configuration, which is the closure of a domain $\Omega \subset \R^3$ and we call $\overline{\Omega}$ the elastic body. We further, assume that the boundary $\Sigma {:=} \partial \Omega$ is sufficiently smooth. If $\vc{\psi}(\vc{x})$ is the deformation map then the material response of the elastic body is described by a stored energy {\rm W} that is a real-valued function of the deformation gradient $\mathbb{F} {:=} \nabla \vc{\psi}$\footnote{We have made the standard assumption that $\mathbb{F} \in \{ \mathbb{M} \in {\rm GL}(\R^3) \mid {\rm det}(\mathbb{M}) > 0 \}$ where ${\rm GL}(\R^3)$ represents the general linear group of order $3$ while the third order tensors will be symmetric in their first two indices. For a more detailed setting the reader can consult \cite{MaHea06}.} and the gradient of the deformation gradient $\mathbb{G}{:=} \nabla\nabla \vc{\psi}$. We denote by $\widetilde{\vc{u}}(\vc{x}){:=} \vc{\psi}(\vc{x}) - \vc{x}$ the displacement and assume that follows some scaling $\widetilde{\vc{u}}(\vc{x}){:=} \alpha \, \vc{u}(\vc{x})$, for some positive constant $\alpha$ (we clarify later on why where we make use of such an $\alpha$). Elementary calculations yield immediately, $\mathbb{F} {=} {\rm I} +  \alpha \nabla \vc{u}$, where ${\rm I}$ is the second order identity tensor.  

{\bf Notation:} To expedite the presentation of our results, here onwards we will make use of the following notation: We will use the Einstein summation for repeated indices unless otherwise stated. Moreover, we will use the symbols $\two$ and $\three$ to indicate second order contractions and third order contractions among tensors, respectively, while $\vc{\epsilon}_{ijk}$ will be the Levi-Civita permutation tensor.

The internal energy of the elastic body is given by, 

\begin{equation}
{\rm E}(\vc{\psi}) = \int_\Omega {\rm W}(\mathbb{F}, \mathbb{G}) d\vc{x},
\end{equation}

where the stored energy satisfies the principle of material objectivity\footnote{${\rm W}(\mathbb{Q}\mathbb{F}, \mathbb{Q}\mathbb{G}) = {\rm W}(\mathbb{F}, \mathbb{G})$ for all $\mathbb{Q} \in {\rm SO}(3)$, $\mathbb{G}$ symmetric in their first two components, and $\mathbb{F} \in \{ \mathbb{M} \in {\rm GL}(\R^3) \mid {\rm det}(\mathbb{M}) > 0) \}$}. The equilibrium equations are derived by computing the first variation of ${\rm E}(\vc{\psi})$ and equate it to the virtual work of some body force field $\vc{g}$ acting through an admissible variation \cite{MaHea06}. Integration by parts, then, gives,

\begin{equation}\label{eq:SG_elasticity}
-{\rm div} \left( \frac{ \partial {\rm W}(\mathbb{F}, \mathbb{G}) }{ \partial \mathbb{F}} - {\rm div} \frac{ \partial {\rm W}(\mathbb{F}, \mathbb{G}) }{ \partial \mathbb{G}} \right) = \vc{g} \text{ in } \Omega,
\end{equation}

Upon using the classical chain rule we can rewrite the above equation as follows,

\begin{equation}
\begin{split}
\mathsf{A}(\mathbb{F},\mathbb{G})[\nabla^{(4)} \psi] + \mathsf{S}_{ij}^{rpq}(\mathbb{F},&\mathbb{G})[\nabla^{(3)} \vc{\psi}] - \mathsf{K}(\mathbb{F},\mathbb{G})[\nabla \nabla \vc{\psi}] + \vc{b}(\nabla^{(3)} \vc{\psi}, \nabla \nabla \vc{\psi}) = \vc{g}
\end{split}
\end{equation}

where 

\begin{equation}
\begin{split}
\mathsf{A}(\mathbb{F},\mathbb{G})[\nabla^{(4)} \vc{\psi}]_i &:= \frac{\partial^2 {\rm W}}{\partial G_{pqr} \partial G_{ijk}} \frac{\partial^4 \psi_p}{\partial x_j \partial x_k \partial x_q \partial x_r},\\
\mathsf{S}(\mathbb{F},\mathbb{G})[\nabla^{(3)} \vc{\psi}]_i &:= -\frac{\partial^2 {\rm W}}{\partial G_{pqr} \partial F_{ij}} \frac{\partial^3 \psi_p}{\partial x_j \partial x_q \partial x_r} + \frac{\partial^2 {\rm W}}{\partial F_{pq} \partial G_{ijk}} \frac{\partial^3 \psi_p}{\partial x_j \partial x_k  \partial x_q},\\
\mathsf{K}(\mathbb{F},\mathbb{G})[\nabla \nabla \vc{\psi}]_i &:= \frac{\partial^2 {\rm W}}{\partial F_{pq} \partial F_{ij}} \frac{\partial^2 \psi_p}{\partial x_j \partial x_q},\\
b_i(\nabla^{(3)} \vc{\psi}, \nabla \nabla \vc{\psi}) &:= \left[\frac{\partial}{ \partial x_j} \frac{\partial^2 {\rm W}}{\partial G_{pqr} \partial G_{ijk}}\right] \frac{\partial^3 \psi_p}{\partial x_k  \partial x_q \partial x_r} + \left[\frac{\partial}{ \partial x_j} \frac{\partial^2 {\rm W}}{\partial F_{pq} \partial G_{ijk}}\right] \frac{\partial^2 \psi_p}{\partial x_k  \partial x_q}.
\end{split}
\end{equation}

Throughout, the work we assume that the uniform strong ellipticity condition, i.e., there exist positive (generic) constants $c_1$ and $c_2$ such that:

\begin{equation} \label{eq:ellipticity1}
c_1 |\vc{w}|^2 |\vc{q}|^4 \le \vc{w} \otimes \vc{q} \otimes \vc{q} \three \mathsf{A}(\mathbb{F}, \mathbb{G})[\vc{w} \otimes \vc{q} \otimes \vc{q}] \le c_2 |\vc{w}|^2 |\vc{q}|^4
\end{equation}

for all $\vc{w}, \vc{q} \in \R^3 - \{ \vc{0} \}$ and for all $(\mathbb{F}, \mathbb{G})$ with $\mathbb{G}$ symmetric in the first two components and $\mathbb{F} \in \{ \mathbb{M} \in {\rm GL}(\R^3) \mid {\rm det}(\mathbb{M}) > 0) \}$. Furthermore, at the reference state we assume that,

\begin{equation} \label{eq:ellipticity2}
c_1 |\vc{w}|^2 |\vc{q}|^2 \le \vc{w} \otimes \vc{q} \two \mathsf{K}(\mathbb{I}, \mathbb{0})[\vc{w} \otimes \vc{q}] \le c_2 |\vc{w}|^2 |\vc{q}|^2
\end{equation}

for all $\vc{w}, \vc{q} \in \R^3 - \{ \vc{0} \}$. Additionally, we assume that the tensor $\mathsf{S}$ belongs to ${\rm L}^\infty (\Omega, \R^{3 \times 3 \times 3 \times 3 \times 3})$.

We linearlize equation \eqref{eq:SG_elasticity} by carrying out a Taylor expansion of the stored energy $\rm W$ around the reference state\footnote{We have added a detailed derivation of the Taylor expansion in the appendix for the readers convenience} $(\mathbb{F}, \mathbb{G})=(\mathbb{I}, \mathbb{0})$ and we obtain the following classical linearlized equations of second-gradient elasticity,  

\begin{gather}
\begin{aligned} \label{eq:derived_model}
&{-}{\rm div}~\tau = \vc{g} \text{ in } \Omega, \\
&\tau {:=} \sigma {-} {\rm div}~\mu \text{ in } \Omega, 
\end{aligned}
\end{gather}

where the quantities $\sigma$ and $\mu$ are related to the deformation and the gradient of the deformation by the following constitutive laws:

\begin{equation} 
\sigma {=} \mathsf{K} \two \nabla \vc{u} {+} \mathsf{S} \three \nabla \nabla \vc{u}, \quad
\mu {=} \mathsf{A} \three \nabla \nabla \vc{u} {+} \mathsf{S} \two \nabla \vc{u},
\end{equation}

which is a mechanical constitutive law up to $\mathcal{O}(\alpha)$ in the expansion and where,

\begin{gather}
\begin{aligned}\label{eq:coeff_sym}
\mathsf{K} &{:=} \frac{ \partial^2 {\rm W} }{ \partial \mathbb{F} \partial \mathbb{F}} \left ( \mathbb{I}, \mathbb{0} \right ), \quad
\mathsf{S} {:=} \frac{ \partial^2 {\rm W} }{ \partial \mathbb{F} \partial \mathbb{G}} \left ( \mathbb{I}, \mathbb{0} \right ), \quad
\mathsf{A} {:=} \frac{ \partial^2 {\rm W} }{ \partial \mathbb{G} \partial \mathbb{G}} \left ( \mathbb{I}, \mathbb{0} \right ).
\end{aligned}
\end{gather}

\section{Background and set up of the problem}

\subsection{Dimensional analysis and scaling}

The elastic body $\overline{\Omega}$ is assumed to be periodic with period $\ell$ and with characteristic length $\rm L$. We define the dimensionless coordinates and displacement,  

\begin{equation} 
\vc{x}^* =  \frac{\vc{x}}{\rm L}, \quad \vc{u}^* (\vc{x}^*) = \frac{\vc{u}(\vc{x})}{\rm L}.
\end{equation}

Moreover, we define the following non-dimensional tensors:

\begin{align}
\mathcal{K} \mathsf{K}^* = \mathsf{K}, \quad \mathcal{S} \mathsf{S}^* = \mathsf{S}, \quad \mathcal{A} \mathsf{A}^* = \mathsf{A}.
\end{align}

where 

\begin{align}
\mathcal{K} {:=} \max_{\vc{z} \in Y_\ell} |\mathsf{K}(\vc{z})|, \quad \mathcal{S} {:=} \max_{\vc{z} \in Y_\ell} |\mathsf{S}(\vc{z}) |, \quad
\mathcal{A} {:=} \max_{\vc{z} \in Y_\ell} |\mathsf{A}(\vc{z})|,
\end{align}

with $Y_\ell {:=} (-\ell/2, \ell/2]^3$ the periodic cell characterizing the body $\Omega$, while $\tau^*{:=}\mathcal{K}^{-1} \tau$ will be the non-dimensional hyperstress.

In generalized continua, there are additional intrinsic lengths related to the microstructure of the material. We refer the reader to reference \cite{Aifantis11} for a modern review on the topic. Since we are interested in modelling chiral microstructures (reference \cite{Lakes01} addresses the modelling of chirality in elastic materials) we will focus our attention on an additional length scale related to chirality. Following the work of references \cite{Forest01}, \cite{Nika21}, \cite{nika2022cosserat} we introduce the subsequent length scales related to the microstructure of the material: 

\begin{equation}\label{scaling}
\mathcal{A} {:=} \mathcal{K} \, \ell^2_{\rm SG}, \quad \mathcal{S} {:=} \mathcal{K} \, \ell^{1/p}_{\rm SG} \, \ell^{1/p'}_{\rm chiral} \quad \text{ where } \quad \frac{1}{p}+\frac{1}{p'} = 1, \quad p, p' \in (1,\infty).
\end{equation}

The scaling \eqref{scaling} provides consistency in the sense that you cannot have chiral effects without having second-gradient effects. However, you can have second-gradient effects without chiral effects. The interplay between $\ell_{\rm SG}$ and $\ell_{\rm chiral}$ is related to the well-posedness of the model, specifically, coercivity. We address this issue in detail in subsequent sections.

The non-dimensional stress in \eqref{eq:derived_model} has the following form,

\begin{gather}
\begin{aligned}
\tau^* 
&{:=} \mathsf{K}^* \two \nabla^{*} \vc{u}^* {+} \left(\frac{\ell_{\rm chiral}}{\rm L}\right)^{1/p'} \left(\frac{\ell_{\rm SG}}{\rm L}\right)^{1/p} \mathsf{S}^{*} \three \nabla^{*}  \nabla^{*}  \vc{u}^* \\
&{-} {\rm div^*} \left( \left( \frac{\ell_{\rm SG}}{\rm L} \right)^2 \mathsf{A}^{*} \three \nabla^{*} \nabla^{*} \vc{u}^* {+} \left(\frac{\ell_{\rm chiral}}{\rm L}\right)^{1/p'} \left(\frac{\ell_{\rm SG}}{\rm L}\right)^{1/p} \mathsf{S}^{*} \two \nabla^{*} \vc{u}^{*} \right),
\end{aligned}
\end{gather}

where the material tensors $\mathsf{K}^*(\vc{x}^*) = \{ \mathsf{K}^*_{jik\ell} (\vc{x}*) \}_{j,i,k,\ell=1}^{3}$, $\mathsf{S}^*(\vc{x}^*)=\{\mathsf{S}_{ji}^{k\ell m*} (\vc{x}*) \}_{j,i,k,\ell,m=1}^{3}$, and $\mathsf{A}^*(\vc{x}^*)=\{\mathsf{A}_{ijk}^{n\ell m*} (\vc{x}*) \}_{j,i,k,n,\ell,m=1}^{3}$ are $Y^*$ periodic with,

\begin{equation}
Y^* {:=} \frac{\ell}{\rm L} Y, \quad Y {:=} \left (-\frac{1}{2}, \frac{1}{2} \right ]^3.
\end{equation}  

Thus, one can generate an $\ep$ periodic problem by defining the non-dimensional number $\ep$ as the ratio of $\ell / {\rm L}$ and let $\ep \to 0$ to obtain an effective medium. However, different cases ought to be considered depending on how the intrinsic length scales $\ell_{\rm chiral}$ and $\ell_{\rm SG}$ scale with $\ell$ and ${\rm L}$, respectively. Here we consider the cases,

\begin{align}
&\ell_{\rm SG} / {\rm L} \sim \ep \quad \text{ and } \quad \ell_{\rm chiral} / \ell \sim \ep^{p'} \label{eq:hs1} \tag{\rm HS 1} \\[1.em]
&\ell_{\rm SG} / {\rm L} \sim 1 \quad \text{ and } \quad \ell_{\rm chiral} / \ell \sim \ep^{p' - 1}.\label{eq:hs2} \tag{\rm HS 2}
\end{align}

We chose to work with the above scalings, primarily, because of their physical interpretation. The \eqref{eq:hs1} scaling indicates that the size of the heterogeneities are comparable to the order of the period \footnote{Recent numerical and experimental work has determined that a micro-to-macro length ratio of $/56$ is sufficient to have strict scale separation and ignore second-gradient effects (i.e. $\ell_{\rm SG} / {\rm L} \approx 1/56$) \cite{Mahmoud2023}.}. The \eqref{eq:hs2} scaling indicates that the size of the heterogeneities are comparable to the characteristic length of the overall domain. Moreover, the chirality scaling has a more general form. However, it cannot be chosen independently of $\ell_{\rm SG}$. The reason being, as we will show in the next section, well-posedness of the model. In our case, the chirality length is (at least) one order smaller compared to the length of second-gradient effects. Naturally, one could consider a different scaling than the one proposed above. We will not address other type of scaling here. Rather we will leave their treatment to future work. Finally, if confusion arises, henceforth, we will omit the $*$ notation for the sake of simplicity and expediency of presentation. 

\subsubsection{Scaling of the stress and hyperstress under {\rm HS 1}}

If $\ell_{\rm chiral}/{\rm \ell}=\ep^{p'}$ then $\ell_{\rm chiral}/{\rm L}=\ep^{p' + 1}$. Hence, the hyperstress becomes,

\begin{gather}
\begin{aligned}
\tau^\ep
{=} \mathsf{K}(\frac{\vc{x}}{\ep}) \two \nabla \vc{u}^\ep &{+} \ep^2 \mathsf{S}(\frac{\vc{x}}{\ep}) \three \nabla \nabla \vc{u}^\ep 
{-} {\rm div} \left( \ep^2 \mathsf{A}(\frac{\vc{x}}{\ep}) \three \nabla \nabla \vc{u}^\ep {+} \ep^2 \mathsf{S}(\frac{\vc{x}}{\ep}) \two \nabla \vc{u}^\ep \right),
\end{aligned}
\end{gather}

where 

\begin{equation}
\sigma^\ep {=} \mathsf{K}(\frac{\vc{x}}{\ep}) \two \nabla \vc{u}^\ep {+} \ep^2 \mathsf{S}(\frac{\vc{x}}{\ep}) \three \nabla \nabla \vc{u}^\ep 
\end{equation} 

and 

\begin{equation}
\mu^\ep {=} \ep^2 \mathsf{A}(\frac{\vc{x}}{\ep}) \three \nabla \nabla \vc{u}^\ep {+} \ep^2 \mathsf{S}(\frac{\vc{x}}{\ep}) \two \nabla \vc{u}^\ep.
\end{equation}

\subsubsection{Scaling of the stress and hyperstress under {\rm HS 2}}

If $\ell_{\rm SG}/{\rm L}=1$ and $\ell_{\rm chiral}/\ell=\ep^{p' - 1}$, then $\ell_{\rm chiral}/{\rm L}=\ep^{p'}$. Hence, the hyperstress becomes,

\begin{gather}
\begin{aligned}
\tau^\ep
{=} \mathsf{K}(\frac{\vc{x}}{\ep}) \two \nabla \vc{u}^\ep &{+} \ep \mathsf{S}(\frac{\vc{x}}{\ep}) \three \nabla \nabla \vc{u}^\ep 
{-} {\rm div} \left( \mathsf{A}(\frac{\vc{x}}{\ep}) \three \nabla \nabla \vc{u}^\ep {+} \ep \mathsf{S}(\frac{\vc{x}}{\ep}) \two \nabla \vc{u}^\ep \right),
\end{aligned}
\end{gather}

where 

\begin{equation}
\sigma^\ep {=} \mathsf{K}(\frac{\vc{x}}{\ep}) \two \nabla \vc{u}^\ep {+} \ep \mathsf{S}(\frac{\vc{x}}{\ep}) \three \nabla \nabla \vc{u}^\ep 
\end{equation} 

and 

\begin{equation}
\mu^\ep {=} \mathsf{A}(\frac{\vc{x}}{\ep}) \three \nabla \nabla \vc{u}^\ep {+} \ep \mathsf{S}(\frac{\vc{x}}{\ep}) \two \nabla \vc{u}^\ep.
\end{equation}

\section{The microscopic model}

We consider an elastic body with periodic microstructure of period $\varepsilon$ occupying a region $\Omega \subset \R^3$. The region $\Omega$ that the body occupies, is assumed to be a uniformly Lipschitz open set (see \cite[Definition 2.65]{DD12}). $Y=(-1/2,1/2]^3$ is the unit cube in $\R^3$, and $\Z^3$ is the set of all $3$--dimensional vectors with integer components. For every positive $\ep$, let $N_\ep$ be the set of all points $\kappa \in \mathbb{Z}^3$ such that $\ep(\kappa+Y)$ is strictly included in $\Omega$. Denote by $T$ be the closure of an open subset in $Y$ with Lipschitz boundary and by $T^\ep_\kappa {:=} \ep (\kappa + T)$ will represent the region containing one of the material phases (see \figref{fig:decomposition}). Hence, we can define the following subsets of $\Omega$:

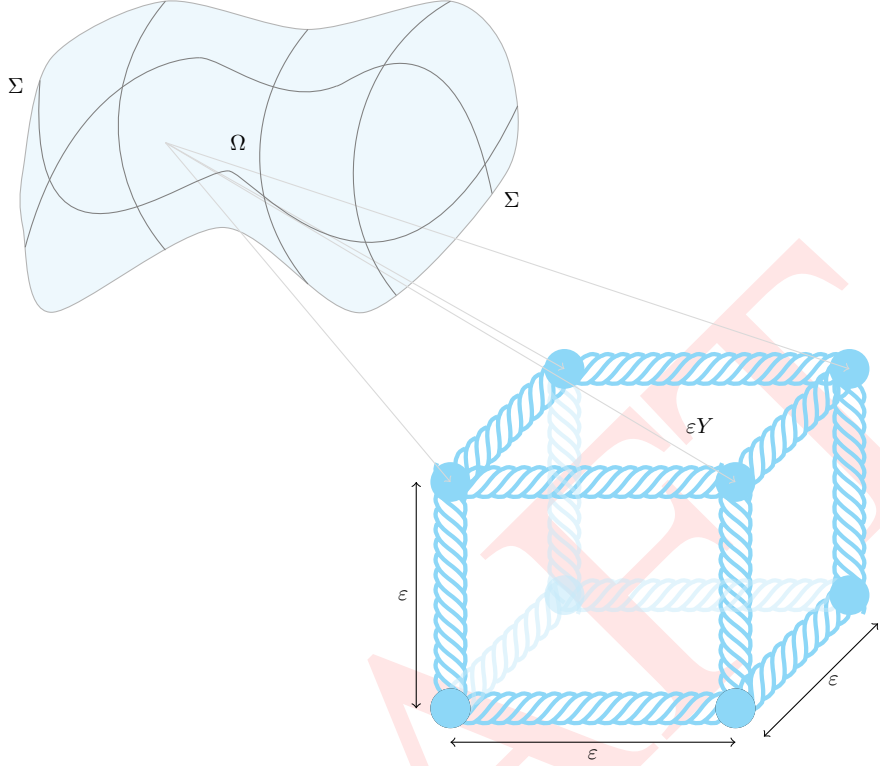
\begin{figure}[!tbh]
\centering
\begin{tikzpicture}[scale=0.75, transform shape, rotate=0, decoration=rope]

%%%% Domain in reference configuration
\draw [fill=cyan!20, opacity=0.3] plot[smooth cycle] coordinates {(-3.5,0.5) (-3,2.5) (-1,3.5) (1.5,3) (4,3.5) (5,2.5) (5,0.5) (2.5,-2) (0,-0.5) (-3,-2) (-3.5,-0.5)};

%\draw (10,1) node[right] {$\mathbb{R}^d$}; 
\draw (0,1) node[right] {$\Omega$}; 
\draw (4.8,0) node[right] {$\Sigma$}; 
\draw (-3.9,2) node[right] {$\Sigma$}; 
 
  %\draw[gray,thick,decorate] (0,0) to (4,-4);
  \draw[cyan!40,ultra thick,decorate,rope width=10pt] (4,-5) to (9,-5);
  \draw[cyan!40,ultra thick,decorate,rope width=10pt] (6,-3) to (11,-3);
  \draw[cyan!40,ultra thick,decorate,rope width=10pt] (4,-5) to (6,-3);
  \draw[cyan!40,ultra thick,decorate,rope width=10pt] (9,-5) to (11,-3);  

  \draw[cyan!40,ultra thick,decorate,rope width=10pt] (4,-9) to (9,-9);
  \draw[cyan!40,ultra thick,decorate,rope width=10pt] (4,-9) to (4,-5);
  \draw[cyan!40,ultra thick,decorate,rope width=10pt] (9,-9) to (9,-5);

  \draw[cyan!40,ultra thick,decorate,rope width=10pt] (11,-3) to (11,-7);
  \draw[cyan!40,ultra thick,decorate,rope width=10pt] (9,-9) to (11,-7);
  
  \draw[cyan!20,ultra thick,decorate,rope width=10pt, opacity=0.5] (6,-3) to (6,-7);
  \draw[cyan!20,ultra thick,decorate,rope width=10pt, opacity=0.5] (6,-7) to (11,-7);
  \draw[cyan!20,ultra thick,decorate,rope width=10pt, opacity=0.5] (6,-7) to (4,-9);
   
  \fill[cyan!40]  (4,-5) circle (10pt) (9,-5) circle (10pt);
  \fill  (4,-9) circle (10pt) (9,-9) circle (10pt);
  \fill[cyan!40]  (4,-9) circle (10pt) (9,-9) circle (10pt);
  \fill[cyan!40]  (6,-3) circle (10pt) (11,-3) circle (10pt);
  \fill[cyan!40]  (11,-7) circle (10pt);
  \fill[cyan!20, opacity=0.6] ( 6,-7) circle (10pt);

\draw [->,black!15](-1, 1.0) -- (4, -5);
\draw [->,black!15](-1, 1.0) -- (9,-5);
\draw [->,black!15](-1, 1.0) -- (6,-3);
\draw [->,black!15](-1, 1.0) -- (11,-3);

\draw [<->] (4,-9.6) to (9,-9.6);
\draw (6.5,-9.6) node[below] {$\varepsilon$};

\draw [<->](9.5,-9.5) -- (11.5,-7.5);
\draw (10.5,-8.5) node[right] {$\varepsilon$};

\draw [<->](3.4,-9) -- (3.4,-5);
\draw (3.4,-7.0) node[left] {$\varepsilon$};

\draw (8,-4.0) node[right] {$\varepsilon Y$};

% Drawing the contours to make body appear 3D
\draw[black!50] (-3.46, -0.85) .. controls (-2.7,2.0) and (-1.0,2.5).. (-0.4,2.5);
\draw[black!50] (-0.4,2.5) .. controls (-0.0,2.5) and (1,1.6) .. (2,2);
\draw[black!50] (2,2) .. controls (2.2,2) and (4,3.7) .. (4.73, 0.1);

\draw[black!50] (-3.2, 2.1) .. controls (-3.5,-2.0) and (-0.6,0.5).. (0.1,0.5);
\draw[black!50] (0.1,0.5) .. controls (0.5,0.5) and (3.0,-3.0) .. (5.18,1.65) ;

%vertical
\draw[color=black!50] (-1,3.5) to [bend right=40] (-1,-0.9);
\draw[color=black!50] (1.5,3) to [bend right=40] (1.5,-1.5);
\draw[color=black!50] (4,3.5) to [bend right=50] (3.05,-1.7);

%%Draw thin grid lines with color 40% gray + 60% white
%\draw [step=1,thin,gray!30] (-6,-10) grid (15,15);
%
%%Draw x and y axis lines
%\draw [->,gray!90] (-11,0) -- (11,0) node [below] {$x$};
%\draw [->,gray!90] (0,-11) -- (0,11) node [left] {$y$};
\end{tikzpicture}
\caption{\small Schematic of the domain $\Omega$ with a (possible) helical type microstructure. One can imagine the helical microstructure re-enforcing the interior of the unit cell which is filled with a ``weak" material where we have assumed perfect transmission conditions across the interphase. Second-gradient elasticity allows for the modelling of domains with helical type microstructures, where they respond to compression by twisting.}
\label{fig:decomposition}
\end{figure}

\[
\Omega_{1\ep} {:=} \displaystyle\bigcup_{\kappa \in N_\ep} T^\ep_\kappa \ ,\quad \Omega_{2\ep} {:=} \Omega \backslash \Omega_{1\ep}, \quad \Omega {:=} \Omega_{1\ep} \cup \Omega_{2\ep}.
\]

The exterior boundary component will be denoted by $\Sigma {:=} \partial \Omega$. We decompose $\Sigma {:=} \Sigma_{\rm 0} \cup \Sigma_{\rm 1}$ with $\Sigma_{\rm 0} \cap \Sigma_{\rm 1} = \emptyset$ a.e.. The vector $\vc{n}$ will be the unit normal on $\Sigma$, pointing in the outward direction. Moreover, thermodynamic stability bounds require that the tensors $\mathsf{K}$, $\mathsf{A}$, and $\mathsf{S}$ possess major symmetries (indicated by the structure of the coefficients in \eqref{eq:coeff_sym}). Furthermore, in addition to the conditions imposed by equations \eqref{eq:ellipticity1} and \eqref{eq:ellipticity2}, we assume that $\mathsf{K}_{jikl} \in {\rm L}^\infty (Y)$, $\mathsf{A}_{ijk}^{n \ell m} \in {\rm L}^\infty (Y)$, $\mathsf{S}_{ji}^{k \ell m } \in {\rm L}^\infty (Y)$ are bounded, measurable functions that can be extended as $Y$-periodic functions to the entirety of $\mathbb{R}^3$ while we reserve the notation for the coefficients, 

\begin{equation}
\mathsf{K} (\frac{\vc{x}}{\ep}) = \mathsf{K} (\vc{y}), \quad \mathsf{S} (\frac{\vc{x}}{\ep}) = \mathsf{S} (\vc{y}), \quad \mathsf{A} (\frac{\vc{x}}{\ep}) = \mathsf{A} (\vc{y})
\end{equation}

where $\vc{y} = \vc{x} / \ep$. In case of isotropy, the above tensors take the following form (see, e.g., \cite{SCh00}, \cite{SDCh01}, \cite{DIso09}),

\begin{equation}
\mathsf{K}_{ijkl} = \lambda \delta_{ij}\delta_{kl} + \mu (\delta_{ik}\delta_{jl} + \delta_{il}\delta_{jk}),     
\end{equation}

\begin{equation}
\mathsf{S}^{klp}_{ji} =
C_8 (\vc{\epsilon}_{ikl}\delta_{jp} + \vc{\epsilon}_{ikp}\delta_{jl} + \vc{\epsilon}_{jkl}\delta_{ip} + \vc{\epsilon}_{jkp}\delta_{il}),      
\end{equation}

\begin{gather}
\begin{aligned}
\mathsf{A}_{ijk}^{lpq} 
=
&
C_3 (\delta_{ij}\delta_{kl}\delta_{pq} + \delta_{ij}\delta_{kp}\delta_{ql} + \delta_{ik}\delta_{jq}\delta_{lp} + \delta_{iq}\delta_{jk}\delta_{lp}) + C_4 \delta_{ij}\delta_{kq}\delta_{lp} \\
+
&
C_5 (\delta_{ik}\delta_{jl}\delta_{pq} + \delta_{ik}\delta_{jp}\delta_{lq} + \delta_{il}\delta_{jk}\delta_{pq} + \delta_{ip}\delta_{jk}\delta_{lq}) + C_6 (\delta_{il}\delta_{jp}\delta_{kq} + \delta_{ip}\delta_{jl}\delta_{kq})\\  
+ 
&
C_7 (\delta_{il}\delta_{jq}\delta_{kp} + \delta_{ip}\delta_{jq}\delta_{kl} + \delta_{iq}\delta_{jl}\delta_{kp} + \delta_{iq}\delta_{jp}\delta_{kl}).     
\end{aligned}
\end{gather}

\subsection{Auxiliary formulas}

For the readers convenience and for the expediency of the our results, we introduce certain formulas that we will make use of in what follows. These formulas can also be found in \cite[Appendix]{Germain73}.

For any sufficiently smooth scalar function $\xi$ defined on $\Sigma$ or on a neighborhood of $\Sigma$ the tangential and normal components of $\nabla\xi$ are,

\begin{equation} \label{eq:grad_decomp}
(\nabla \xi)_\tau {=} - \vc{n} \times (\vc{n} \times \nabla \xi) {=} \nabla \xi - (\nabla \xi)_n \vc{n}, \quad (\nabla \xi)_n {:=} \nabla \xi \cdot \vc{n}.
\end{equation}

Moreover, we introduce the surface gradient of $\xi$ using the projection operator $\Pi {:=} {\rm I} - \vc{n} \otimes \vc{n}$. 

\[
\nabla_s \xi {=} ({\rm I} - \vc{n} \otimes \vc{n}) \nabla \xi {=} \Pi \nabla \xi.
\]

Thus, we can write down a useful integration by parts on surfaces formula,

\begin{equation} \label{eq:surface_parts}
\int_\Sigma \nabla_s \xi \,ds = \int_\Sigma \xi ({\rm div}\vc{n}) \vc{n} \, ds + \int_{\partial \Sigma} \jump{\xi \vc{\nu}} \, d\ell,
\end{equation}

where 

\[
\nu_i=\vc{\epsilon}_{ijk} t_j n_k, \quad i=1,2,3,
\]

is a component of the unit normal vector on $\partial \Sigma$ and tangent to $\Sigma$, $t_j$ is a component of the unit tangent vector to $\partial \Sigma$. Lastly, we remark, the jump term on \eqref{eq:surface_parts} is on a ridge, i.e., the line on $\Sigma$ where the tangent plane of $\Sigma$ is discontinuous. The above formulas are used with a high degree of frequency in emulsions and capillary fluids (see, e.g., \cite{NVDCDS16}). We refer the reader to the appendix of reference \cite{Germain73}, \cite{Germain73II} for an excellent exposition of the above formulae and related topics.

Using the above formulas and notation, the heterogeneous medium is then be characterized by the following system (written component-wise) for $i=1,2,3$:

\begin{gather}
\begin{aligned} \label{eq:model}
-&\partial_{x_j}\tau^\ep_{ij} = g_i &\text{ in }& \Omega, \\
&\tau^\ep_{ij} = \sigma^\ep_{ij} - \partial_{x_k} \mu^\ep_{ijk} &\text{ in }& \Omega, \\
&( \sigma^\ep_{ij} - \partial_{x_k} \mu^\ep_{ijk} ) n_j - \Pi_{q\ell} \partial_{x_\ell} ( \mu^\ep_{ijk} n_k \Pi_{qj} ) = 0 &\text{ on }& \Sigma_{\rm 1},\\
&\mu^\ep_{ijk} n_k n_j = 0 &\text{ on }& \Sigma_{\rm 1},\\
&u^\ep_i = 0 &\text{ on }& \Sigma_{\rm 0},\\
&\frac{\partial u^\ep_i}{\partial \vc{n}}=0 &\text{ on }& \Sigma_{\rm 0},\\
&\jump{\mu^\ep_{ijk} n_k \nu_j} = 0 &\text{ on }& \partial \Sigma_1,
\end{aligned}
\end{gather} 

where $g_i$ is a component some appropriately scaled body force that belongs in ${\rm L}^2(\Omega)$ and $\nu_i$ is a component of the outward unit normal to $\partial \Sigma$, for $i=1,2,3$. 

Given that the boundary conditions for a second-gradient material are not as conventional as the boundary conditions for a classical Cauchy material we write out explicitly what mechanical forces they represent on the elastic body following references \cite{Germain73, Germain73II}. Thus, besides the classical homogeneous Dirichlet boundary condition, we also have:

\begin{itemize}
\item[-] Surface traction: $( \sigma^\ep_{ij} - \partial_{x_k} \mu^\ep_{ijk} ) n_j - \Pi_{q\ell} \partial_{x_\ell} ( \mu^\ep_{ijk} n_k \Pi_{qj}),$
\item[-] A normal double traction: $\mu^\ep_{ijk} n_k n_j,$
\item[-] A line traction: $\jump{\mu^\ep_{ijk} n_k \nu_j}.$
\end{itemize}

\subsection{Variational formulation}

The primary setting for this work is the Sobolev space ${\rm H}^2(\Omega, \R^3)$, the space of functions $\vc{u} : \Omega \mapsto \R^3$ such that each coordinate is twice weakly differentiable and all the first and second partial derivatives are in ${\rm L}^2(\Omega)$ and the subspace ${\rm H}^2_{\Sigma_{\rm 0}}(\Omega, \R^3)$ which consists functions that vanish along with their derivatives on the part of the boundary of $\Sigma$, $\Sigma_{\rm 0}$ (see, e.g. \cite{adams2003sobolev}).  

The space ${\rm H}^2(\Omega, \R^3)$ is a Hilbert space with norm,

\begin{equation} \label{eq:norm_1}
\norm{\vc{u}}_{{\rm H}^2(\Omega, \R^3)} = \left( \norm{\vc{u}}^2_{\rm L^2(\Omega, \R^3)} + \norm{\nabla \vc{u}}^2_{\rm L^2(\Omega, \R^{3\times3})} +\norm{\nabla \nabla \vc{u}}^2_{\rm L^2(\Omega, \R^{3\times3\times3})} \right)^{1/2}.
%\sum_{|\alpha| \le 2} \norm{\partial^\alpha \vc{u} }^2_{{\rm L}^2(\Omega)} \right)^{1/2},
\end{equation}

Hence, if we multiply \eqref{eq:model} by $\vc{v} \in \{ C^\infty(\overline{\Omega}, \R^3) \mid \vc{v} = 0, \,\, \nabla \vc{v} = 0 \text{ on } \Sigma_{\rm 0} \}$ and integrate by parts, then we obtain:

\begin{equation}
{-}\int_{\Sigma_{\rm 1}} (\sigma^\ep_{ij} {-} \partial_{x_k} \mu^\ep_{ijk}) n_j v_i \, ds + \int_{\Omega} (\sigma^\ep_{ij} {-} \partial_{x_k}\mu^\ep_{ijk}) \partial_{x_j}v_i \, d\vc{x} = \int_\Omega g_i v_i \, d\vc{x}.
\end{equation}

A second integration by parts of the second term on the second integral gives,

\begin{gather}
\begin{aligned}
{-}\int_{\Sigma_{\rm 1}} (\sigma^\ep_{ij} {-} \partial_{x_k} \mu^\ep_{ijk}) n_j v_i \, ds &+ \int_{\Omega} \sigma^\ep_{ij} \partial_{x_j} v_i \, d\vc{x}\\
&+ \int_{\Omega} \mu^\ep_{ijk} \partial^2_{x_j x_k} v_i \, d\vc{x} - \int_{\Sigma_{\rm 1}} \mu^\ep_{ijk} n_k \partial_{x_j} v_i \, ds = \int_\Omega g_i v_i \, d\vc{x}.
\end{aligned}
\end{gather}

The last term on the left hand side of the above equation requires a second integration by parts. However, we first decompose it into its normal and tangential component (see equation \eqref{eq:grad_decomp}) as follows,

\begin{equation}
\int_{\Sigma_{\rm 1}} \mu^\ep_{ijk} n_k \partial_{x_j} v_i \, ds = \int_{\Sigma_{\rm 1}} \mu^\ep_{ijk} n_k n_j n_l \partial_{x_l} v_i \, ds + \int_{\Sigma_{\rm 1}} \mu^\ep_{ijk} n_k \Pi_{lj} \partial_{x_l} v_i \, ds 
\end{equation}

A second integration by parts on surfaces (see equation \eqref{eq:surface_parts}) for the last term on the right hand side of the above equation gives,

\begin{gather}
\begin{aligned}
\int_{\Sigma_{\rm 1}} \mu^\ep_{ijk} n_k \Pi_{lj} \partial_{x_l} v_i \, ds 
&= \int_{\Sigma_{\rm 1}} ( \mu^\ep_{ijk} n_k \Pi_{qj} ({\rm div}~\vc{n}) n_q - \Pi_{ql}\partial_{x_l} ( \mu^\ep_{ijk} n_k \Pi_{qj} ) ) v_i \, ds \\
&- \int_{\partial \Sigma_{\rm 1}} \jump{\mu^\ep_{ijk} n_k \nu_j v_i} d\ell.
\end{aligned}
\end{gather}

We remark immediately,

\begin{equation}
\mu^\ep_{ijk} n_k \Pi_{qj} ({\rm div}~\vc{n}) n_q = (\mu^\ep_{ijk} n_k n_j - \mu^\ep_{ijk} n_k n_q n_j n_q) ({\rm div}~\vc{n}) = 0. 
\end{equation}

Hence, using a density argument, the variational formulation of \eqref{eq:model} is: Find $\vc{u}^\ep \in {\rm H}^2_{\Sigma_{\rm 0}}(\Omega, \R^3)$ such that,

\begin{equation}\label{eq:weak}
\int_{\Omega} \sigma^\ep_{ij} \partial_{x_j} v_i \, d\vc{x} + \int_{\Omega} \mu^\ep_{ijk} \partial^2_{x_j x_k} v_i \, d\vc{x} = \int_\Omega g_i v_i \, d\vc{x},
\end{equation}

for all $\vc{v} \in {\rm H}^2_{\Sigma_{\rm 0}}(\Omega, \R^3)$. 

\subsection{Existence and uniqueness}

Denote by,

\begin{equation}
{\rm B}[\vc{u}^\ep, \vc{v}] {:=} \int_{\Omega} \sigma^\ep_{ij} \partial_{x_j} v_i \, d\vc{x} + \int_{\Omega} \mu^\ep_{ijk} \partial^2_{x_j x_k} v_i \, d\vc{x}.
\end{equation}

The form $\rm B$ is evidently a bilinear form that is continuous in the weak topology of ${\rm H}^2 \times {\rm H}^2$ and it remains to show coercivity in order to apply the Lax-Milgram theorem. 

\subsubsection{Coercivity in {\rm HS} 1}

Using the strong ellipticity conditions in \eqref{eq:ellipticity1} and \eqref{eq:ellipticity2} together with Cauchy's inequality with $\delta$ we obtain,

\begin{gather}
\begin{aligned} \label{eq:HS1_coercive}
\kappa_1 \ep^2 \norm{\nabla\nabla\vc{u}^\ep}_{\rm L^2(\Omega, \R^{3\times3\times3)}}^2 
+ 
&
c_1 \norm{\nabla\vc{u}^\ep}_{\rm L^2(\Omega, \R^{3\times3})}^2 \\
\le 
&
{\rm B}[\vc{u}^\ep, \vc{u}^\ep] - 2 \ep^2 \int_\Omega  \mathsf{S}^{klm}_{ij}(\frac{\vc{x}}{\ep}) \frac{\partial^2 u^\ep_k}{\partial x_m \partial x_l} \frac{\partial u^\ep_i}{\partial x_j} \, d\vc{x} \\
\le 
&
{\rm B}[\vc{u}^\ep, \vc{u}^\ep] + 2 \ep^2 \int_\Omega  |\nabla\nabla\vc{u}^\ep| | \nabla\vc{u}^\ep| \, d\vc{x} \\
\le 
&
{\rm B}[\vc{u}^\ep, \vc{u}^\ep] + 2 \ep^2 \delta \norm{\nabla\nabla\vc{u}^\ep}^2 + \frac{\ep^2}{2\delta} \norm{\nabla\vc{u}^\ep}^2.
\end{aligned}
\end{gather}

Thus, 

\begin{equation} \label{eq:HS1_bound}
(\kappa_1 - 2 \delta) \ep^2 \norm{\nabla\nabla\vc{u}^\ep}_{\rm L^2(\Omega, \R^{3\times3\times3})}^2 + (c_1- \frac{\ep^2}{2\delta}) \norm{\nabla\vc{u}^\ep}_{\rm L^2(\Omega, \R^{3\times3})}^2 \le {\rm B}[\vc{u}^\ep, \vc{u}^\ep]. 
\end{equation}

By selecting $\delta < \kappa_1/4$, using Poincar\'e's inequality in ${\rm H^1_{\Sigma_{\rm 0}}(\Omega, \R^3)}$, and then using the smallness of $\ep$ to guarantee $(c_1- \frac{2 \ep^2}{\kappa_1})=:c>0$, we ensure the desired ellipticity:

\begin{equation}
\min \{\kappa_1/2, c\} \, c_\Omega \, \ep^2 \norm{\vc{u}^\ep}^2_{\rm H^2(\Omega, \R^3)} \le {\rm B}[\vc{u}^\ep, \vc{u}^\ep]. 
\end{equation}

Additionally, starting with \eqref{eq:HS1_bound}, by utilizing Poincar\'e's inequality in ${\rm H^1_{\Sigma_{\rm 0}}(\Omega, \R^3)}$ one can obtain the following estimate  for the solution (under {\rm HS 1}):

\begin{equation} \label{eq:HS1_estimate}
\Big( \norm{\vc{u}^\ep}^2_{\rm H^1_{\Sigma_{\rm 0}}(\Omega, \R^3)} + \ep^2 \norm{\nabla\nabla \vc{u}^\ep}^2_{\rm L^2(\Omega, \R^{3 \times 3 \times 3})} \Big)^{1/2} \le {\rm const.} \norm{\vc{g}}_{\rm L^2(\Omega, \R^3)},
\end{equation}

for some generic constant independent of $\ep$.

\subsubsection{Coercivity in {\rm HS} 2}

Coercivity in this case can be shown in exactly the same way as in {\rm HS} 1. We simply write it down and omit the details,

\begin{equation} \label{eq:HS2_bound}
\min \{\kappa_1/2, c\} \, c_\Omega \, \norm{\vc{u}^\ep}^2_{\rm H^2(\Omega, \R^3)} \le {\rm B}[\vc{u}^\ep, \vc{u}^\ep]. 
\end{equation}

Naturally, a similar estimate can be obtained under the scheme {\rm HS 2},

\begin{equation} \label{eq:HS2_estimate}
 \norm{\vc{u}^\ep}_{\rm H^2(\Omega, \R^3)}  \le {\rm const.} \norm{\vc{g}}_{\rm L^2(\Omega, \R^3)},
\end{equation}

again, the constant is a generic constant independent of $\ep$. Hence, by the Lax-Milgram lemma, under both schemes, there exists a unique solution $\vc{u}^\ep \in \rm{H}^2_{\Sigma_{\rm 0}}(\Omega, \R^3)$ to \eqref{eq:weak}. We also refer the  reader to the works of \cite{eremeyev2021strong}, \cite{eremeyev2023ellipticity} regarding coercivity of different generalized continua. 

\section{Homogenization of the second-gradient continuum}

\subsection{The periodic unfolding}

We define the following domain decompositions (see~\cite{CDG02, Dam05, CDG08, CDG18}):

\begin{equation}
K_\ep^- {:=} \left \{ \ell \in \Z^3 \mid \ep (\ell + Y) \subset \overline{\Omega} \right \}, \quad
\Omega_\ep^- {:=} {\rm int} \left ( \cup_{\ell \in K_\ep^-} \ep (\ell + Y) \right), \quad
\Lambda_\ep^- {:=} \Omega \backslash \Omega_\ep^-.
\end{equation}

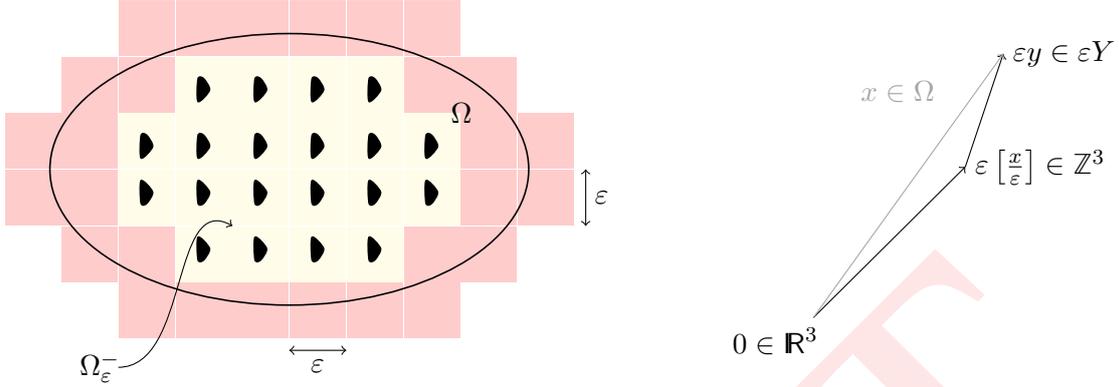
\begin{figure}[!tbh]
\begin{minipage}{.57\linewidth}
\begin{tikzpicture}[every node/.style={minimum size=.745cm-\pgflinewidth, outer sep=0pt}, scale=1.5]

\draw [<->] (2.6,0) -- (2.6,-0.5);
\draw [semithick,black] (2.5,-0.25) node [right] {$\ep$};

\draw [<->] (0,-1.6) -- (0.5,-1.6);
\draw [semithick,black] (0.25,-1.5) node [below] {$\ep$};

%inside domain
\node[fill=yellow!10] at (0.25,0.75) {};
\node[fill=yellow!10] at (0.75,0.75) {};

\node[fill=yellow!10] at (0.25,0.25) {};
\node[fill=yellow!10] at (0.75,0.25) {};
\node[fill=yellow!10] at (1.25,0.25) {};

\node[fill=yellow!10] at (0.25,-0.25) {};
\node[fill=yellow!10] at (0.75,-0.25) {};
\node[fill=yellow!10] at (1.25,-0.25) {};

\node[fill=yellow!10] at (0.25,-0.75) {};
\node[fill=yellow!10] at (0.75,-0.75) {};

%left side

\node[fill=yellow!10] at (-0.25,0.75) {};
\node[fill=yellow!10] at (-0.75,0.75) {};

\node[fill=yellow!10] at (-0.25,0.25) {};
\node[fill=yellow!10] at (-0.75,0.25) {};
\node[fill=yellow!10] at (-1.25,0.25) {};

\node[fill=yellow!10] at (-0.25,-0.25) {};
\node[fill=yellow!10] at (-0.75,-0.25) {};
\node[fill=yellow!10] at (-1.25,-0.25) {};

\node[fill=yellow!10] at (-0.25,-0.75) {};
\node[fill=yellow!10] at (-0.75,-0.75) {};

%boundary squares 
\node[fill=red!20] at (0.25,1.25) {};
\node[fill=red!20] at (0.75,1.25) {};
\node[fill=red!20] at (1.25,1.25) {};
\node[fill=red!20] at (1.25,0.75) {};
\node[fill=red!20] at (1.75,0.75) {};
\node[fill=red!20] at (1.75,0.25) {};
\node[fill=red!20] at (2.25,0.25) {};
\node[fill=red!20] at (2.25,-0.25) {};
\node[fill=red!20] at (1.75,-0.25) {};
\node[fill=red!20] at (1.75,-0.75) {};
\node[fill=red!20] at (1.25,-0.75) {};
\node[fill=red!20] at (1.25,-1.25) {};
\node[fill=red!20] at (0.75,-1.25) {};
\node[fill=red!20] at (0.25,-1.25) {};

%symmetry
\node[fill=red!20] at (-0.25,1.25) {};
\node[fill=red!20] at (-0.75,1.25) {};
\node[fill=red!20] at (-1.25,1.25) {};
\node[fill=red!20] at (-1.25,0.75) {};
\node[fill=red!20] at (-1.75,0.75) {};
\node[fill=red!20] at (-1.75,0.25) {};
\node[fill=red!20] at (-2.25,0.25) {};
\node[fill=red!20] at (-2.25,-0.25) {};
\node[fill=red!20] at (-1.75,-0.25) {};
\node[fill=red!20] at (-1.75,-0.75) {};
\node[fill=red!20] at (-1.25,-0.75) {};
\node[fill=red!20] at (-1.25,-1.25) {};
\node[fill=red!20] at (-0.75,-1.25) {};
\node[fill=red!20] at (-0.25,-1.25) {};

% Draw x and y axis lines
%\draw [->] (-3,0) -- (3,0) node [below] {$x$};
%\draw [->] (0,-3) -- (0,3) node [left] {$y$};

%Drawing of the domain. It is not necessarilly and ellipse but this is the best shape I could find.
\draw [semithick,black] (0,0) ellipse (2.1 and 1.2);

\draw [semithick,black] (1.75,0.5) node [left] {$\Omega$};

%\draw [semithick,black] (2.40,2.0) node [right] {$\Omega^{+}_\ep$};
%\draw[->] (2.5,2.0) to [out=180, in=100] (1.25,1.25);

\draw [semithick,black] (-1.4,-1.75) node [left] {$\Omega^{-}_\ep$};
\draw[->] (-1.50,-1.75) to [out=0, in=150] (-0.5,-0.5);

\draw[semithick,black,fill=black] plot [smooth cycle] coordinates {(0.25,0.3) (0.3,0.2) (0.2,0.1) (0.2,0.3)}; 
\draw[semithick,black,fill=black] plot [smooth cycle] coordinates {(0.75,0.3) (0.8,0.2) (0.7,0.1) (0.7,0.3)};
\draw[semithick,black,fill=black] plot [smooth cycle] coordinates {(1.25,0.3) (1.3,0.2) (1.2,0.1) (1.2,0.3)};
%\draw[semithick,gray,fill=gray] plot [smooth cycle] coordinates {(1.7,0.2) (1.75,0.3) (1.7,0.3) (1.7,0.2)};

\draw[semithick,black,fill=black] plot [smooth cycle] coordinates {(0.25,0.8) (0.3,0.7) (0.2,0.6) (0.2,0.8)};
\draw[semithick,black,fill=black] plot [smooth cycle] coordinates {(0.75,0.8) (0.8,0.7) (0.7,0.6) (0.7,0.8)};
%\draw[semithick,gray,fill=gray] plot [smooth cycle] coordinates {(1.2,0.73) (1.25,0.81) (1.2,0.77) (1.2,0.65)};

\draw[semithick,black,fill=black] plot [smooth cycle] coordinates {(-0.25,0.3) (-0.2,0.2) (-0.3,0.1) (-0.3,0.3)}; 
\draw[semithick,black,fill=black] plot [smooth cycle] coordinates {(-0.75,0.3) (-0.7,0.2) (-0.8,0.1) (-0.8,0.3)};
\draw[semithick,black,fill=black] plot [smooth cycle] coordinates {(-1.25,0.3) (-1.2,0.2) (-1.3,0.1) (-1.3,0.3)};
%\draw[semithick,gray,fill=gray] plot [smooth cycle] coordinates {(1.7,0.2) (1.75,0.3) (1.7,0.3) (1.7,0.2)};

\draw[semithick,black,fill=black] plot [smooth cycle] coordinates {(-0.25,0.8) (-0.2,0.7) (-0.3,0.6) (-0.3,0.8)};
\draw[semithick,black,fill=black] plot [smooth cycle] coordinates {(-0.75,0.8) (-0.7,0.7) (-0.8,0.6) (-0.8,0.8)};
%\draw[semithick,gray,fill=gray] plot [smooth cycle] coordinates {(-1.2,0.73) (-1.25,0.81) (-1.2,0.77) (-1.2,0.65)};

\draw[semithick,black,fill=black] plot [smooth cycle] coordinates {(-0.25,-0.3) (-0.2,-0.2) (-0.3,-0.1) (-0.3,-0.3)}; 
\draw[semithick,black,fill=black] plot [smooth cycle] coordinates {(-0.75,-0.3) (-0.7,-0.2) (-0.8,-0.1) (-0.8,-0.3)};
\draw[semithick,black,fill=black] plot [smooth cycle] coordinates {(-1.25,-0.3) (-1.2,-0.2) (-1.3,-0.1) (-1.3,-0.3)};
%\draw[semithick,gray,fill=gray] plot [smooth cycle] coordinates {(1.7,0.2) (1.75,0.3) (1.7,0.3) (1.7,0.2)};

\draw[semithick,black,fill=black] plot [smooth cycle] coordinates {(-0.25,-0.8) (-0.2,-0.7) (-0.3,-0.6) (-0.3,-0.8)};
\draw[semithick,black,fill=black] plot [smooth cycle] coordinates {(-0.75,-0.8) (-0.7,-0.7) (-0.8,-0.6) (-0.8,-0.8)};
%\draw[semithick,gray,fill=gray] plot [smooth cycle] coordinates {(-1.2,0.73) (-1.25,0.81) (-1.2,0.77) (-1.2,0.65)};

\draw[semithick,black,fill=black] plot [smooth cycle] coordinates {(0.25,-0.3) (0.3,-0.2) (0.2,-0.1) (0.2,-0.3)}; 
\draw[semithick,black,fill=black] plot [smooth cycle] coordinates {(0.75,-0.3) (0.8,-0.2) (0.7,-0.1) (0.7,-0.3)};
\draw[semithick,black,fill=black] plot [smooth cycle] coordinates {(1.25,-0.3) (1.3,-0.2) (1.2,-0.1) (1.2,-0.3)};
%\draw[semithick,gray,fill=gray] plot [smooth cycle] coordinates {(1.7,0.2) (1.75,0.3) (1.7,0.3) (1.7,0.2)};

\draw[semithick,black,fill=black] plot [smooth cycle] coordinates {(0.25,-0.8) (0.3,-0.7) (0.2,-0.6) (0.2,-0.8)};
\draw[semithick,black,fill=black] plot [smooth cycle] coordinates {(0.75,-0.8) (0.8,-0.7) (0.7,-0.6) (0.7,-0.8)};
\end{tikzpicture}
\end{minipage}
\hspace{.05\linewidth}
\begin{minipage}{.45\linewidth}
\begin{tikzpicture}[every node/.style={minimum size=.51cm-\pgflinewidth, outer sep=0pt}]
% Draw thin grid lines with color 40% gray + 60% white
%\draw [step=2.0,thin,gray!60] (1,1) grid (5,5);

\draw [semithick,black] (-0.5,0.0) node [below] {$0 \in \R^3$};
\draw [->](0,0) -- (2,2);
\draw [->] (2,2) -- (2.5,3.5);
\draw [->,gray!70!white] (0,0) -- (2.5,3.5);

\draw [semithick,black] (2.0,2.0) node [right] {$\ep \left[ \frac{x}{\ep} \right] \in \Z^3$};
\draw [semithick,black] (2.5,3.5) node [right] {$\ep y \in \ep Y$};
\draw [semithick,gray!70!white] (0.5,3.0) node [right] {$x \in \Omega$};
\end{tikzpicture}
\end{minipage}
\caption{\small Schematic decomposition of the domain and definition of the unfolding operator on a periodic grid.}
\label{fig:unfolding}
\end{figure}

Let $[\vc{z}]_Y = (\lfloor z_1\rfloor, \lfloor z_2 \rfloor, \lfloor z_3 \rfloor)$ denote the integer part of $\vc{z} \in \R^3$ and denote by $\{\vc{z}\}_Y$ the difference $\vc{z} - [\vc{z}]_Y$ which belongs to $Y$. Regarding our multiscale problem that depends on a small length parameter $\ep > 0$, we can decompose any $\vc{x} \in \R^3$ using the maps $[\cdot]_Y: \R^3 \mapsto \Z^3$ and $\{\cdot\}_Y: \R^3 \mapsto Y$ the following way (see~\figref{fig:unfolding} (right)),

\begin{equation}
\vc{x} = \ep \left( \left[ \frac{\vc{x}}{\ep}\right]_Y + \left\{ \frac{\vc{x}}{\ep} \right\}_Y \right).
\end{equation}

For any Lebesgue measurable function $\varphi$ on $\Omega$ we define the periodic unfolding operator by,

\begin{equation}\label{eq:two_comp}
\mathcal{T}_\ep (\varphi)(\vc{x}, \vc{y}) =
\begin{cases}
\varphi \left( \ep \left[ \frac{\vc{x}}{\ep} \right]_Y + \ep \vc{y} \right) & \text{ for a.e. } (\vc{x}, \vc{y}) \in \Omega_{\ep}^- \times Y \\
0 & \text{ for a.e. } (\vc{x}, \vc{y}) \in \Lambda_{\ep}^- \times Y.
\end{cases}
\end{equation}

\begin{prop} \label{prop:properties}
For any $p \in [1,+\infty)$ the unfolding operator $\mathcal{T}_\ep: {\rm L}^p(\Omega) \mapsto {\rm L}^p(\Omega \times Y)$ is linear, continuous, and has the following properties:

\begin{itemize}

\item[I.] 
$\mathcal{T}_\ep (\varphi \, \psi) = \mathcal{T}_\ep (\varphi) \, \mathcal{T}_\ep (\psi)$ for every pair of Lebesgue measurable functions $\varphi$, $\psi$ on $\Omega$

\item[II.] 
For every $\varphi \in {\rm L}^1(\Omega)$ we have,
\begin{equation}
\frac{1}{|Y|} \int_{\Omega \times Y} \mathcal{T}_\ep(\varphi)(\vc{x}, \vc{y}) \, d\vc{x} \, d\vc{y} = \int_{\Omega_{\ep}^-} \varphi(\vc{x}) \, d\vc{x} = \int_{\Omega_{\ep}} \varphi(\vc{x}) \, d\vc{x} - \int_{\Lambda_{\ep}^-} \varphi(\vc{x}) \, d\vc{x}
\end{equation}

\item[III.] 
$\norm{\mathcal{T}_\ep(\varphi)}_{{\rm L}^p(\Omega \times Y)} \le |Y|^{1/p} \norm{\varphi}_{{\rm L}^p(\Omega)}$ for every $\varphi \in {\rm L}^p(\Omega)$

\item[IV.] 
$\mathcal{T}_\ep(\varphi) \to \varphi$ strongly in ${\rm L}^p(\Omega \times Y)$ for $\varphi \in {\rm L}^p(\Omega)$ as $\ep \to 0$

\item[V.] 
If $\{ \varphi_\ep \}_{\ep}$ is a sequence in ${\rm L}^p(\Omega)$ such that $\varphi_\ep \to \varphi$ strongly in ${\rm L}^p(\Omega)$, then $\mathcal{T}_\ep (\varphi_\ep) \to \varphi$ strongly in ${\rm L}^p(\Omega \times Y)$

\item[VI.] 
If $\varphi \in L^p(Y)$ is Y-periodic and $\varphi_\ep(\vc{x}) = \varphi \left( \frac{\vc{x}}{\ep} \right)$ then $\mathcal{T}_\ep(\varphi_\ep) \to \varphi$ strongly in ${\rm L}^p(\Omega \times Y)$ as $\ep \to 0$

\item[VII.] 
If $\phi_\ep \weak \phi$ in ${\rm H}^{1}(\Omega)$ then there exists an non-relabelled subsequence and a $\hat{\phi} \in {\rm L}^2(\Omega; {\rm H}^1_{\rm per}(Y))$ such that 
	\begin{itemize}
		\item[a.]	$\mathcal{T}_\ep(\phi_\ep) \weak \phi$ in ${\rm L}^2(\Omega; {\rm H}^1(Y))$
		\item[b.]	$\mathcal{T}_\ep(\nabla \phi_\ep) \weak \nabla_x \phi + \nabla_y \hat{\phi}$ in ${\rm L}^2(\Omega \times Y, \R^3)$
	\end{itemize}

\item[VIII.] Let $\phi_\ep \in {\rm H}^{1}(\Omega)$ and assume that $\{\phi_\ep\}_{\ep}$ is a bounded sequence in ${\rm L}^2(\Omega)$ satisfying $\ep \, \norm{\nabla \phi_\ep}_{{\rm L}^2(\Omega; \R^d)} \le c$ ($c$ is a constant independent of $\ep$) then there exists an non-relabelled subsequence and a $\hat{\phi} \in {\rm L}^2(\Omega; {\rm H}^1_{\rm per}(Y))$ such that 
	\begin{itemize}
		\item[a.]	$\mathcal{T}_\ep(\phi_\ep) \weak \hat{\phi}$ in ${\rm L}^2(\Omega; {\rm H}^1(Y))$
		\item[b.]	$\ep \, \mathcal{T}_\ep(\nabla \phi_\ep) \weak \nabla_y \hat{\phi}$ in ${\rm L}^2(\Omega \times Y)$
	\end{itemize}

\item[IX.] 
If $\phi_\ep \weak \phi$ in ${\rm H}^{2}(\Omega)$ then there exists an non-relabelled subsequence and a $\hat{\phi} \in {\rm L}^2(\Omega; {\rm H}^2_{\rm per}(Y))$ such that 
	\begin{itemize}
		\item[a.]	$\mathcal{T}_\ep(\phi_\ep) \weak \phi$ in ${\rm L}^2(\Omega; {\rm H}^2(Y))$
		\item[b.]	$\mathcal{T}_\ep(\nabla \phi_\ep) \weak \nabla_x \phi$ in ${\rm L}^2(\Omega \times Y, \R^3)$
		\item[c.]    $\mathcal{T}_\ep(\nabla\nabla \phi_\ep) \weak \nabla_x\nabla_x \phi + \nabla_y\nabla_y \hat{\phi}$ in ${\rm L}^2(\Omega \times Y, \R^{3 \times 3})$
	\end{itemize}
\end{itemize} 
\end{prop}

The proof of \propref{prop:properties} can be found in reference \cite{CDG08}. We draw the readers attention to property {\it IX}. which deals with unfolding higher gradients (and shows the true usefulness of the unfolding method). The proof of property {\it IX}. can be found in reference \cite[Theorem 3.6, pg. 1603]{CDG08}.

\subsection{Presentation and discussion of the main results}

In this section we present the main results of our work, discuss their significance and consequences, and address how they compare/differ with results in the current literature. Their, respective, proofs are postponed until \secref{sec:proofs}.

\begin{restatable}{theorem}{hsone}\label{T1:hs1}
If $\vc{u}^\ep \in {\rm H}^2_{\Sigma_{\rm 0}}(\Omega, \R^3)$ is the solution to \eqref{eq:weak} then, under the {\rm HS 1} scheme, there exist $\vc{u}^0 \in {\rm H}^1_{\Sigma_{\rm 0}}(\Omega; \R^3)$, $\hat{\vc{u}} \in {\rm L}^2(\Omega; {\rm H}^2_{\rm per}(Y; \R^3))$ such that,

\begin{align}
&\mathcal{T}_\ep (\vc{u}^\ep) \weak \vc{u}^0 \text{ in } {\rm L}^2(\Omega; {\rm H}^2(Y; \R^3)), \label{eq:unfold_u_HS1}\\
&\mathcal{T}_\ep(\nabla \vc{u}^\ep) \weak \nabla_x \vc{u}^0 + \nabla_y \hat{\vc{u}} \text{ in } {\rm L}^2(\Omega; {\rm H}^1(Y; \R^{3 \times 3})), \label{eq:unfold_grad_u_HS1} \\
&\mathcal{T}_\ep(\ep \nabla\nabla \vc{u}^\ep) \weak \nabla_y \nabla_y \hat{\vc{u}} \text{ in } {\rm L}^2(\Omega \times Y; \R^{3 \times 3 \times 3}), \label{eq:unfold_sec_grad_u_HS1} 
\end{align}

and $(\vc{u}^0, \hat{\vc{u}})$ is the unique solution set of, 

\begin{gather} 
\begin{aligned}\label{eq:effective_HS1}
\int_{\Omega \times Y} \mathsf{K}(\vc{y}) (\nabla_x \vc{u}^0 + &\nabla_y \hat{\vc{u}}) \two (\nabla_x \vc{V} + \nabla_y \overline{\vc{W}}) \, d\vc{y} d\vc{x} \\
+
&
\int_{\Omega \times Y} \mathsf{A}(\vc{y}) \nabla_y\nabla_y \hat{\vc{u}} \three \nabla_y\nabla_y \overline{\vc{W}} \, d\vc{y} \, d\vc{x} = \int_{\Omega \times Y} \vc{g} \cdot \vc{V} \, d\vc{y} d\vc{x},
\end{aligned}
\end{gather}

for all $\vc{V} \in {\rm H}^1_{\Sigma_{\rm 0}}(\Omega; \R^3)$ and $\overline{\vc{W}} \in {\rm L}^2(\Omega; {\rm H}^2(Y; \R^3))$. Furthermore, \eqref{eq:effective_HS1} is equivalent to the following,

\begin{align}\label{effective:HS1}
&\int_{\Omega} \mathsf{K}^{\rm eff} \nabla_x \vc{u}^0 \two \nabla_x \vc{V} d\vc{x} = \int_{\Omega} \vc{g} \one \vc{V} \, d\vc{x},
\end{align}

if $\hat{u}_i (\vc{x}, \vc{y}) = \frac{\partial u^0_\alpha}{\partial_{x_\beta}}(\vc{x}) \varphi^{\alpha \beta}_i(\vc{y}) + \kappa_i(\vc{x})$, for $i=1,2,3$, and we select $\overline{\vc{W}} \equiv \vc{0}$. Here, 

\begin{equation}
\mathsf{K}^{\rm eff}_{ij\alpha \beta} := \int_Y \mathsf{K}_{ijkl}(\vc{y}) \left( \delta_{\alpha k} \delta_{\beta l} + \frac{\partial}{\partial_{y_l}} \varphi^{\alpha \beta}_k \right) \, d\vc{y},
\end{equation}

where $\vc{\varphi}^{\alpha \beta}$ is the unique solution (up to a constant) to,

\begin{gather}\label{hs1:corr_ref}
\left\{
\begin{aligned}
-&{\rm div_y} \left( \mathsf{K} \two \left( \vc{e}_\alpha \otimes \vc{e}_{\beta} + \nabla_y \vc{\varphi}^{\alpha \beta} \right) - {\rm div_y} \left( \mathsf{A} \three \nabla_y\nabla_y \vc{\varphi}^{\alpha \beta} \right) \right) = \vc{0} \text{ in } Y, \\ 
& \vc{\varphi}^{\alpha \beta}(\vc{y}) \text{ is } Y-\text{periodic}.
\end{aligned}
\right.
\end{gather}
\end{restatable}

The model in \thmref{T1:hs1} approximates a second-gradient heterogeneous material with chiral effects by a homogeneous classical linear elastic material. Thus, through homogenization we arrive to a non-local constitutive law where the non-locality is a due to the scaling \eqref{eq:hs1}. There are two main differences from the models that exist in the literature: First, $\vc{u}^\ep$ possesses higher regularity due to Sobolev embedding theory. Indeed, the solution $\vc{u}^\ep$ of \eqref{eq:weak} under \eqref{eq:hs1} is (H\"older) continuous $C^{0,\lambda}(\Omega, \R^3)$, for all $\lambda \in (0,1/2)$ since,

\[
	{\rm H}^2(\Omega, \R^3) \hookrightarrow C^{0,\lambda}(\overline{\Omega}, \R^3) \quad \forall \, \lambda \in (0,1/2),
\]

with the embedding being compact \cite[Theorem 2.84, pg. 98]{DD12}. Second, the structure of the corrector problem in \eqref{hs1:corr_ref}. The corrector solutions are constructed using second-gradient theory and depend both on the material tensor $\mathsf{K}$ as well as the tensor $\mathsf{A}$. Moreover, when no second-gradient effects are present, i.e., the tensor $\mathsf{A}$ is identically zero, we recover the classical corrector problem as in references \cite{Bens78, SP80, BP89, CioDo00, MV10}. Additionally, the corrector solution inherits the same regularity as $\vc{u}^\ep$ and, with it, all the attributes that make it more appealing from the point of view of computational mechanics, i.e., H\"older continuity.

\begin{restatable}{theorem}{hstwo}\label{T2:hs2}
If $\vc{u}^\ep \in {\rm H}^2_{\Sigma_{\rm 0}}(\Omega, \R^3)$ is the solution to \eqref{eq:weak} then, under the {\rm HS 2} scheme, there exist $\vc{u}^0 \in {\rm H}^2_{\Sigma_{\rm 0}}(\Omega, \R^3)$, $\hat{\vc{u}} \in {\rm L}^2(\Omega; {\rm H}^2_{\rm per}(Y; \R^3))$ such that,

\begin{align}
%&\vc{u}^\ep \weak \vc{u}^0 \text{ in } H^1(\Omega; \R^d)\\
&\mathcal{T}_\ep (\vc{u}^\ep) \weak \vc{u}^0 \text{ in } {\rm L}^2(\Omega; {\rm H}^2(Y; \R^3)), \label{eq:unfold_u_HS2}\\
&\mathcal{T}_\ep(\nabla \vc{u}^\ep) \weak \nabla_x \vc{u}^0 \text{ in } {\rm L}^2(\Omega; {\rm H}^1(Y; \R^{3 \times 3})), \label{eq:unfold_grad_u_HS2} \\
&\mathcal{T}_\ep(\nabla\nabla \vc{u}^\ep) \weak \nabla_x\nabla_x \vc{u}^0 + \nabla_y\nabla_y \hat{\vc{u}} \text{ in } {\rm L}^2(\Omega \times Y; \R^{3 \times 3 \times 3}), \label{eq:unfold_sec_grad_u_HS2} 
\end{align}

and $(\vc{u}^0, \hat{\vc{u}})$ is the unique solution set of, 

\begin{gather}\label{effective:HS2_1}
\begin{aligned}
&\int_{\Omega \times Y} \mathsf{K}(\vc{y}) \nabla_x \vc{u}^0 \two \nabla_x \vc{V}  \, d\vc{y} \, d\vc{x}\\ 
+ 
&\int_{\Omega \times Y} \mathsf{A}(\vc{y}) ( \nabla_x\nabla_x \vc{u}^0 + \nabla_y\nabla_y \hat{\vc{u}}) \three (\nabla_x\nabla_x \vc{V} + \nabla_y\nabla_y \overline{\vc{W}}) \, d\vc{y} \, d\vc{x}\\ 
= &\int_{\Omega \times Y} \vc{g} \one \vc{V} \, d\vc{y} \, d\vc{x},
\end{aligned}
\end{gather}

for all $\vc{V} \in {\rm H}^2_{\Sigma_{\rm 0}}(\Omega, \R^3)$ and $\overline{\vc{W}} \in {\rm L}^2(\Omega; {\rm H}^2(Y; \R^3))$. Furthermore, \eqref{effective:HS2_1} is equivalent to the following,

\begin{align}\label{effective:HS2_2}
&\int_{\Omega} \mean{\mathsf{K}}_Y \nabla_x \vc{u}^0 \two \nabla_x \vc{V} d\vc{x} + \int_{\Omega} \mathsf{A}^{\rm eff} \nabla_x\nabla_x\vc{u}^0 \three \nabla_x\nabla_x \vc{V} \, d\vc{x} = \int_{\Omega} \vc{g} \one \vc{V} \, d\vc{x},
\end{align}

if $\hat{u}_i (\vc{x}, \vc{y}) = \frac{\partial^2 u^0_\alpha (\vc{x})}{\partial_{x_\beta} \partial_{x_\gamma}} w^{\alpha \beta \gamma}_i(\vc{y}) + \kappa_i(\vc{x})$, for $i=1,2,3$, and we select $\overline{\vc{W}} \equiv \vc{0}$. Here, 

\begin{equation}
\mathsf{(A^{\rm eff})}^{{ijk}}_{\alpha \beta \gamma} := \int_Y \mathsf{A}^{ijk}_{n\ell m}(\vc{y}) \left( \delta_{\alpha n} \delta_{\beta m} \delta_{\gamma \ell} + \frac{\partial^2}{\partial_{y_m} \partial_{y_\ell}}w^{\alpha \beta \gamma}_n \right) \, d\vc{y},
\end{equation}

where $\vc{w}^{\alpha \beta \gamma}$ is the unique solution (up to a constant) to,

\begin{gather}\label{hs2:corr_ref}
\left\{
\begin{aligned}
-&{\rm div_y} \left( {\rm div_y} \left(  \mathsf{A} \three \left( \vc{e}_\alpha \otimes \vc{e}_{\beta} \otimes \vc{e}_{\gamma} + \nabla_y\nabla_y \vc{w}^{\alpha \beta \gamma} \right) \right) \right)= \vc{0} \text{ in } Y, \\ 
& \vc{w}^{\alpha \beta \gamma}(\vc{y}) \text{ is } Y-\text{periodic}.
\end{aligned}
\right.
\end{gather}
\end{restatable}

The results of \thmref{T2:hs2}, to our knowledge, are new in their entirety. First, the effective problem \eqref{effective:HS2_2} is of second-gradient type where the effective coefficients are computed using the sixth order tensor $\mathsf{A}$ while the fourth order tensor $\mathsf{K}$ is simply averaged over the unit cell $Y$. Moreover, we draw the readers attention to the structure of the corrector problem in \eqref{hs2:corr_ref} and how it differs from the corrector problem in \eqref{hs1:corr_ref}. It is immediate, that problem \eqref{hs2:corr_ref} uses three different unit ``directional" basis vectors $\vc{e}_\alpha, \vc{e}_{\beta}, \vc{e}_{\gamma}$ instead of the usual two unit ``directional" basis vectors as is standard in the classical theory of elasticity. Furthermore, the same regularity properties, as in the first case, are retained in \thmref{T2:hs2} both for $\vc{u}^0$ and the corrector solution.

Lastly, we remark that the vastly different limit problems obtained under the schemes \eqref{eq:hs1} and \eqref{eq:hs2}, respectively, are solely due to the internal lengths, $\ell_{\rm SG}$ and $\ell_{\rm chiral}$, that second-gradient theory introduces. Namely, when the size of the heterogeneities is comparable with the length of the period then we obtain an effective linear elastic material (with higher corrector regularity as a byproduct). When the size of the heterogeneities is comparable with the overall length of the domain (when scale separation is not possible) then the second-gradient effects are retained on the macroscale and the structure of the corrector problem changes considerably. However, the ${\rm H}^2$ regularity of the solution and the corrector is preserved.

\subsection{Proofs of the main results} \label{sec:proofs}
\subsubsection{Proof of \thmref{T1:hs1}}

\hsone*
\begin{proof}
Using \eqref{eq:HS1_estimate} and \propref{prop:properties} $VII.$ we obtain \eqref{eq:unfold_u_HS1}--\eqref{eq:unfold_grad_u_HS1}. To obtain \eqref{eq:unfold_sec_grad_u_HS1} apply \propref{prop:properties} $IX.$ with $\phi_\ep = \nabla \vc{u}^\ep$ and the result follows. 

We now proceed by unfolding \eqref{eq:weak}, under the {\rm HS 1} scheme, and apply \propref{prop:properties} properties $I., II., \text{ and } {\color{red} VI}.$, to obtain,

\begin{align} \label{eq:HS1_unfolded}
\int_{\Omega {\times} Y} &\left( \mathsf{K}_{ijkl}(\vc{y}) \mathcal{T}_\ep (\frac{\partial u^\ep_k}{\partial x_l}) \mathcal{T}_\ep (\frac{ \partial v_i}{\partial x_j}) {+} \ep^2 \mathsf{S}_{ij}^{klm} (\vc{y}) \mathcal{T}_\ep(\frac{\partial^2 u^\ep_k}{\partial x_m \partial x_l}) \mathcal{T}_\ep(\frac{\partial v_i}{\partial x_j}) \right) \, d\vc{y} d\vc{x} \nonumber \\
{+} 
\int_{\Omega {\times} Y} &\Big( \ep^2 \mathsf{A}_{nlm}^{ijk}(\vc{y}) \mathcal{T}_\ep (\frac{\partial^2 u^\ep_n}{\partial x_l \partial x_m}) \mathcal{T}_\ep ( \frac{\partial^2 v_i}{\partial x_j \partial x_k}) \\ \nonumber
&{+} \ep^2 \mathsf{S}_{nl}^{ijk}(\vc{y}) \mathcal{T}_\ep ( \frac{\partial u^\ep_n}{\partial x_l} ) \mathcal{T}_\ep ( \frac{\partial^2 v_i}{\partial x_j \partial x_k}) \Big) \, d\vc{y} d\vc{x}
= \int_{\Omega {\times} Y} \mathcal{T}_\ep (g_i) \mathcal{T}_\ep (v_i) \, d\vc{y} d\vc{x}, \nonumber
\end{align}

Set $\vc{v}:=\vc{V}(\vc{x})$ to be any test function $\vc{V} \in C_0^\infty(\Omega; \R^3)$ in \eqref{eq:HS1_unfolded}. Taking the limit as $\ep \to 0$ and using the properties of the unfolding operator \eqref{eq:unfold_u_HS1}--\eqref{eq:unfold_sec_grad_u_HS1} we obtain,

\begin{gather} 
\begin{aligned}\label{eq:HS1_unfolded_x}
\int_{\Omega {\times} Y} \mathsf{K}(\vc{y}) (\nabla_x \vc{u}^0 {+} \nabla_y \hat{\vc{u}}) \two \nabla_x \vc{V} \, d\vc{y} d\vc{x} = \int_{\Omega \times Y} \vc{g} \cdot \vc{V} \, d\vc{y} d\vc{x},
\end{aligned}
\end{gather}

Select now test functions of the form $\vc{v}=\vc{v}^\ep:=\ep \, U(\vc{x}) \, \vc{W} \left( \frac{\vc{x}}{\ep} \right)$ where $U \in C_0^\infty(\Omega)$ and $\vc{W} \in {\rm H}^2_{\rm per}(Y, \R^3)$. It is clear that $\vc{v}^\ep \to \vc{0}$ in ${\rm L}^2(\Omega, \R^3)$. Moreover, we have,
 
\begin{equation}
\frac{\partial v^\ep_i}{\partial {x_j}} = \ep \frac{\partial U}{\partial x_j}(\vc{x}) W_i(\frac{\vc{x}}{\ep}) {+} U(\vc{x}) \frac{\partial W_i}{\partial y_j}(\frac{\vc{x}}{\ep}),
\end{equation}

\begin{gather}
\begin{aligned}
\frac{\partial^2 v^\ep_i}{\partial x_j \partial x_k} {=} \ep \frac{\partial^2 U}{\partial x_j \partial x_k}(\vc{x}) W_i(\frac{\vc{x}}{\ep}) 
&{+} \frac{\partial U}{\partial x_j}(\vc{x}) \frac{\partial W_i}{\partial y_k}(\frac{\vc{x}}{\ep}) \\ 
&{+} \frac{\partial U}{\partial x_k}(\vc{x}) \frac{\partial W_i}{\partial y_j}(\frac{\vc{x}}{\ep}) {+} \frac{1}{\ep}U(\vc{x}) \frac{ \partial^2 W_i}{\partial y_j \partial y_k}(\frac{\vc{x}}{\ep}).
\end{aligned}
\end{gather}  
 
Thus, as $\ep \to 0$, we have $\mathcal{T}_\ep (v^\ep_i) \to 0$ in ${\rm L}^2(\Omega \times Y)$  $\mathcal{T}_\ep (\partial_{x_j} v^\ep_i) \weak \nabla_y \overline{W}_i(\vc{x},\vc{y})$ in ${\rm L}^2(\Omega \times Y)$, and $\mathcal{T}_\ep (\ep \,\partial^2_{x_j x_k} v^\ep_i) \weak \partial^2_{y_j y_k} \overline{W}_i(\vc{x},\vc{y})$ in ${\rm L}^2(\Omega \times Y)$ where $\overline{W}_i(\vc{x},\vc{y}):=U(\vc{x}) \, W_i(\vc{y})$. Hence, if in the unfolded expression \eqref{eq:HS1_unfolded} use the above test function we obtain,

\begin{gather} 
\begin{aligned}\label{eq:HS1_unfolded_y}
\int_{\Omega \times Y} \mathsf{K}(\vc{y}) (\nabla_x \vc{u}^0 + &\nabla_y \hat{\vc{u}}) \two \nabla_y \overline{\vc{W}} \, d\vc{y} d\vc{x} \\
+
&
\int_{\Omega \times Y} \mathsf{A}(\vc{y}) \nabla_y\nabla_y \hat{\vc{u}} \three \nabla_y\nabla_y \overline{\vc{W}} \, d\vc{y} \, d\vc{x} = 0,
\end{aligned}
\end{gather}

Thus, adding \eqref{eq:HS1_unfolded_x} and \eqref{eq:HS1_unfolded_y} we obtain,

\begin{gather} 
\begin{aligned}\label{eq:HS1_effective}
\int_{\Omega \times Y} \mathsf{K}(\vc{y}) (\nabla_x \vc{u}^0 + &\nabla_y \hat{\vc{u}}) \two (\nabla_x \vc{V} + \nabla_y \overline{\vc{W}}) \, d\vc{y} d\vc{x} \\
+
&
\int_{\Omega \times Y} \mathsf{A}(\vc{y}) \nabla_y\nabla_y \hat{\vc{u}} \three \nabla_y\nabla_y \overline{\vc{W}} \, d\vc{y} \, d\vc{x} = \int_{\Omega \times Y} \vc{g} \cdot \vc{V} \, d\vc{y} d\vc{x},
\end{aligned}
\end{gather}

By the density of $C_0^{\infty}(\Omega) \otimes {\rm H}^2_{\rm per}(Y; \R^3)$ in ${\rm L}^2(\Omega; {\rm H}^2_{\rm per}(Y; \R^3))$ the result holds for all $\overline{\vc{W}}(\vc{x},\vc{y}) \in {\rm L}^2(\Omega; {\rm H}^2_{\rm per}(Y; \R^3))$. 

If in \eqref{eq:HS1_effective} select $\vc{V} = \vc{0}$, then we can see that $\hat{\vc{u}}$ depends \AM{ linearly} on $\nabla_x \vc{u}^0$. Hence, the form of $\hat{\vc{u}}$ looks as follows:

\begin{equation}\label{eq:HS1_hat_u}
\hat{u}_i (\vc{x}, \vc{y}) = \frac{\partial u^0_\alpha}{\partial_{x_\beta}} (\vc{x}) \varphi^{\alpha \beta}_i(\vc{y}) + \kappa_i(\vc{x}),
\end{equation}

where the corrector $\vc{\varphi}^{\alpha\beta}$ is the local solution satisfying \AM{the next boundary-value problem}

\begin{gather} \label{eq:HS1_local1}
\left\{
\begin{aligned}
-&{\rm div_y} \left( \mathsf{K} \two \left( \vc{e}_\alpha \otimes \vc{e}_{\beta} + \nabla_y \vc{\varphi}^{\alpha \beta} \right) - {\rm div_y} \left( \mathsf{A} \three \nabla_y\nabla_y \vc{\varphi}^{\alpha \beta} \right) \right) = \vc{0} \text{ in } Y, \\ 
& \vc{\varphi}^{\alpha \beta}(\vc{y}) \text{ is } Y-\text{periodic}.
\end{aligned}
\right.
\end{gather}

Equivalently, we can formulate \eqref{eq:HS1_local1} in its weak form: Find $\vc{\varphi}^{\alpha\beta} \in {\rm H}^2_{\rm per}(Y,\R^3)$ such that 

\begin{equation}\label{eq:HS1_weak_local}
\int_Y \left( \mathsf{K}\vc{e}_\alpha \otimes \vc{e}_{\beta} \two \nabla_y\vc{\phi} + \mathsf{K} \nabla_y \vc{\varphi}^{\alpha \beta} \two \nabla_y \vc{\phi} + \mathsf{A}\nabla_y \nabla_y \vc{\varphi}^{\alpha \beta} \three \nabla_y \nabla_y \vc{\phi} \right) \, d\vc{y} = 0
\end{equation}

for all $\vc{\phi} \in {\rm H}^2_{\rm per}(Y,\R^3)$. \AM{The} existence and uniqueness (up to a constant) \AM{of a weak solution to}  \eqref{eq:HS1_weak_local} follows from the Lax-Milgram Lemma over the space ${\rm H}^2_{\rm per}(Y, \R^3)$. 

Returning to \eqref{eq:HS1_effective} and substituting $\overline{\vc{W}}=\vc{0}$ and $\hat{\vc{u}}$ from \eqref{eq:HS1_hat_u} we obtain,

\begin{align}
&\int_{\Omega} \mathsf{K}^{\rm eff} \nabla_x \vc{u}^0 \two \nabla_x \vc{V} d\vc{x} = \int_{\Omega} \vc{g} \one \vc{V} \, d\vc{x},
\end{align}

where, 

\begin{equation}
\mathsf{K}^{\rm eff}_{ij\alpha \beta} := \int_Y \mathsf{K}_{ijkl}(\vc{y}) \left( \delta_{\alpha k} \delta_{\beta l} + \frac{\partial}{\partial_{y_l}} \varphi^{\alpha \beta}_k \right) \, d\vc{y}.
\end{equation}

If we define $\sigma^{\rm eff} {:=} \mathsf{K}^{\rm eff} \two \nabla_x \vc{u}^0$ then $\sigma^{\rm eff} = (\sigma^{\rm eff})^\top$ is precisely the Cauchy stress in the theory of classical linear elasticity. This completes the proof.
\end{proof}

\subsubsection{Proof of \thmref{T2:hs2}}

\hstwo*

\begin{proof}
Using \eqref{eq:HS2_estimate} and \propref{prop:properties} $IX.$ we obtain (up to a subsequence) the \AM{convergences stated}  in \eqref{eq:unfold_u_HS2}--\eqref{eq:unfold_sec_grad_u_HS2}. We now proceed by unfolding \eqref{eq:weak}, under the {\rm HS 2} scheme. To this end, we apply \propref{prop:properties} properties $I., II., \text{ and } {\color{red}VI}.$, to obtain

\begin{align} \label{eq:HS2_unfolded}
&\int_{\Omega \times Y} \left( \mathsf{K}_{ijkl}(\vc{y}) \mathcal{T}_\ep (\frac{\partial u^\ep_k}{\partial x_l}) \mathcal{T}_\ep (\frac{ \partial v_i}{\partial x_j}) + \ep \mathsf{S}_{ij}^{klm} (\vc{y}) \mathcal{T}_\ep(\frac{\partial^2 u^\ep_k}{\partial x_m \partial x_l}) \mathcal{T}_\ep(\frac{\partial v_i}{\partial x_j}) \right) \, d\vc{y} d\vc{x} \nonumber \\
+ 
&\int_{\Omega \times Y} \left(  \mathsf{A}_{nlp}^{ijk}(\vc{y}) \mathcal{T}_\ep (\frac{\partial^2 u^\ep_n}{\partial x_l \partial x_p}) \mathcal{T}_\ep ( \frac{\partial^2 v_i}{\partial x_j \partial x_k}) {+} \ep \mathsf{S}_{nl}^{ijk}(\vc{y}) \mathcal{T}_\ep ( \frac{\partial u^\ep_n}{\partial x_l} ) \mathcal{T}_\ep ( \frac{\partial^2 v_i}{\partial x_j \partial x_k}) \right) \, d\vc{y} d\vc{x} \\ \nonumber
= 
&\int_{\Omega \times Y} \mathcal{T}_\ep (g_i) \mathcal{T}_\ep (v_i) \, d\vc{y} d\vc{x}. \nonumber
\end{align}

Set $\vc{v}:=\vc{V}(\vc{x})$ to be any test function $\vc{V} \in C_0^\infty(\Omega; \R^3)$ in \eqref{eq:HS2_unfolded}. Taking the limit as $\ep \to 0$ and using the properties of the unfolding operator \eqref{eq:unfold_u_HS2}--\eqref{eq:unfold_sec_grad_u_HS2} we obtain,

\begin{gather}
\begin{aligned} \label{eq:HS2_unfolded_x}
&\int_{\Omega \times Y} \mathsf{K}(\vc{y}) \nabla_x \vc{u}^0 \two \nabla_x \vc{V}  \, d\vc{y} \, d\vc{x}\\ 
+ 
&\int_{\Omega \times Y} \mathsf{A}(\vc{y}) ( \nabla_x\nabla_x \vc{u}^0 + \nabla_y\nabla_y \hat{\vc{u}}) \three \nabla_x\nabla_x \vc{V}  \, d\vc{y} \, d\vc{x}\\ 
= &\int_{\Omega \times Y} \vc{g} \one \vc{V} \, d\vc{x}.
\end{aligned}
\end{gather}

Select now test functions of the form $\vc{v}=\vc{v}^\ep:=\ep^2 \, U(\vc{x}) \, \vc{W} \left( \frac{\vc{x}}{\ep} \right)$ where $U \in C_0^\infty(\Omega)$ and $\vc{W} \in {\rm H}^2_{\rm per}(Y, \R^3)$. \AM{We note} that $\vc{v}^\ep \to \vc{0}$ in ${\rm L}^2(\Omega, \R^3)$. Moreover, we have
 
\begin{equation}
\frac{\partial v^\ep_i}{\partial {x_j}} = \ep^2 \frac{\partial U}{\partial x_j}(\vc{x}) W_i(\frac{\vc{x}}{\ep}) {+} \ep U(\vc{x}) \frac{\partial W_i}{\partial y_j}(\frac{\vc{x}}{\ep}),
\end{equation}

\begin{gather}
\begin{aligned}
\frac{\partial^2 v^\ep_i}{\partial x_j \partial x_k} {=} \ep^2 \frac{\partial^2 U}{\partial x_j \partial x_k}(\vc{x}) W_i(\frac{\vc{x}}{\ep}) 
&{+} \ep \frac{\partial U}{\partial x_j}(\vc{x}) \frac{\partial W_i}{\partial y_k}(\frac{\vc{x}}{\ep}) \\ 
&{+} \ep \frac{\partial U}{\partial x_k}(\vc{x}) \frac{\partial W_i}{\partial y_j}(\frac{\vc{x}}{\ep}) {+} U(\vc{x}) \frac{ \partial^2 W_i}{\partial y_j \partial y_k}(\frac{\vc{x}}{\ep}).
\end{aligned}
\end{gather}  
 
Thus, as $\ep \to 0$, it yields $\mathcal{T}_\ep (\partial_{x_j} v^\ep_i) \to 0$ in ${\rm L}^2(\Omega \times Y)$ and $\mathcal{T}_\ep (\partial^2_{x_j x_k} v^\ep_i) \to \partial^2_{y_j y_k} \overline{W}_i(\vc{x},\vc{y})$ in ${\rm L}^2(\Omega \times Y)$ for $\overline{W}_i(\vc{x},\vc{y}):=U(\vc{x}) \, W_i(\vc{y})$. Hence, we use the above test functions in the unfolded expression \eqref{eq:HS2_unfolded} to obtain, 

\begin{gather}
\begin{aligned} \label{eq:HS2_unfolded_y}
\int_{\Omega \times Y} \mathsf{A}(\vc{y}) ( \nabla_x\nabla_x \vc{u}^0 + \nabla_y\nabla_y \hat{\vc{u}}) \three  \nabla_y\nabla_y \overline{\vc{W}} \, d\vc{y} \, d\vc{x} = 0.
\end{aligned}
\end{gather}

Adding \eqref{eq:HS2_unfolded_x} and \eqref{eq:HS2_unfolded_y}, we obtain,

\begin{gather}\label{eq:HS2_effective}
\begin{aligned}
&\int_{\Omega \times Y} \mathsf{K}(\vc{y}) \nabla_x \vc{u}^0 \two \nabla_x \vc{V}  \, d\vc{y} \, d\vc{x}\\ 
+ 
&\int_{\Omega \times Y} \mathsf{A}(\vc{y}) ( \nabla_x\nabla_x \vc{u}^0 + \nabla_y\nabla_y \hat{\vc{u}}) \three (\nabla_x\nabla_x \vc{V} + \nabla_y\nabla_y \overline{\vc{W}}) \, d\vc{y} \, d\vc{x}\\ 
= &\int_{\Omega} \vc{g} \one \vc{V} \, d\vc{x},
\end{aligned}
\end{gather}

Once again, by the density of $C_0^{\infty}(\Omega) \otimes {\rm H}^2_{\rm per}(Y; \R^3)$ in ${\rm L}^2(\Omega; {\rm H}^2_{\rm per}(Y; \R^3))$ the result holds for all $\overline{\vc{W}}(\vc{x},\vc{y}) \in {\rm L}^2(\Omega; {\rm H}^2_{\rm per}(Y; \R^3))$. 

\AM{Proceeding in a similar fashion as for the case} {\rm HS 1}, if  we select in \eqref{eq:HS2_effective} $\vc{V} = \vc{0}$, then we can see that $\hat{\vc{u}}$ depends  linearly on $\nabla_x \nabla_x \vc{u}^0$. Hence, the structure of $\hat{\vc{u}}$ looks as follows,

\begin{equation}\label{eq:HS2_hat_u}
\hat{u}_i (\vc{x}, \vc{y}) = \frac{\partial^2 u^0_\alpha}{\partial_{x_\beta} \partial_{x_\gamma}}(\vc{x}) w^{\alpha \beta \gamma}_i(\vc{y}) + {\rm P}_i(\vc{x}),
\end{equation}

where ${\rm P}_i(\vc{x})$ is a linear polynomial in the variable $y$ and the corrector $\vc{w}^{\alpha\beta\gamma}$ is the local solution satisfying the following problem,

\begin{gather}\label{eq:HS2_local1}
\left\{
\begin{aligned}
-&{\rm div_y} \left( {\rm div_y} \left(  \mathsf{A} \three \left( \vc{e}_\alpha \otimes \vc{e}_{\beta} \otimes \vc{e}_{\gamma} + \nabla_y\nabla_y \vc{w}^{\alpha \beta \gamma} \right) \right) \right)= \vc{0} \text{ in } Y, \\ 
& \vc{w}^{\alpha \beta \gamma}(\vc{y}) \text{ is } Y-\text{periodic}.
\end{aligned}
\right.
\end{gather}

Equivalently, we can formulate \eqref{eq:HS2_local1} in its weak form: Find $\vc{w}^{\alpha\beta\gamma} \in {\rm H}^2_{\rm per}(Y,\R^3)$ such that,

\begin{gather}
\begin{aligned}
\int_Y \left(  \mathsf{A} \vc{e}_\alpha \otimes \vc{e}_{\beta} \otimes \vc{e}_{\gamma} \three \nabla_y \nabla_y \vc{\xi} +  \mathsf{A} \nabla_y\nabla_y \vc{w}^{\alpha \beta \gamma} \three \nabla_y \nabla_y \vc{\xi} \right) \, d\vc{y} = 0.
\end{aligned}
\end{gather}

The existence and uniqueness (up to a rigid displacement) of a weak solution follows based on the Lax-Milgram Lemma. This is straightforward as the Poincar\'e's inequality holds for the quotient space ${\rm H}^2(Y) / \mathcal{P}$, where we designate $\mathcal{P}$ to be the space of linear polynomials (see, e.g. \cite{necas1967methodes}).

We return now to \eqref{eq:HS2_effective}. Substituting $\overline{\vc{W}}=\vc{0}$ and $\hat{\vc{u}}$ from \eqref{eq:HS2_hat_u} we obtain,

\begin{align}
&\int_{\Omega} \mean{\mathsf{K}}_Y \nabla_x \vc{u}^0 \two \nabla_x \vc{V} d\vc{x} + \int_{\Omega} \mathsf{A}^{\rm eff} \nabla_x\nabla_x\vc{u}^0 \three \nabla_x\nabla_x \vc{V} \, d\vc{x} = \int_{\Omega} \vc{g} \one \vc{V} \, d\vc{x},
\end{align}

where,

\begin{equation}
\mean{\mathsf{K}}_Y {:=} \int_{Y} \mathsf{K}(\vc{y}) \, d\vc{y},
\end{equation}

\begin{equation}
\mathsf{(A^{\rm eff})}^{{ijk}}_{\alpha \beta \gamma} := \int_Y \mathsf{A}^{ijk}_{n\ell p}(\vc{y}) \left( \delta_{\alpha n} \delta_{\beta p} \delta_{\gamma \ell} + \frac{\partial^2}{\partial_{y_p} \partial_{y_\ell}}w^{\alpha \beta \gamma}_n \right) \, d\vc{y}.
\end{equation}

\AM{This completes} the proof.
\end{proof}

\begin{remn}
The coefficient $\mathsf{A}^{\rm eff}$ is precisely the coefficient provided phenomenologically by references \cite{ME68}, \cite{Germain73}, however, in our case it is exactly computable based on volume fraction and morphology of the microstructure.  
\end{remn}

\subsubsection{Recovery of an effective second-gradient theory}

The statement of \thmref{T2:hs2} \AM{points out a key aspect -- we are dealing macroscopically with} a second-gradient material (see \eqref{effective:HS2_2}). In this section, we derive the associated \AM{partial differential equations} with its boundary conditions in the sense of distributions and show that they form a complete set of equillibrium equations for the second-gradient theory of \cite{ME68} equivalent to the system given by \cite{Germain73}.

We begin with,

\begin{align} \label{hs2:weak}
&\int_{\Omega} \mean{\mathsf{K}}_Y \nabla_x \vc{u}^0 \two \nabla_x \vc{V} d\vc{x} + \int_{\Omega} \mathsf{A}^{\rm eff} \nabla_x\nabla_x\vc{u}^0 \three \nabla_x\nabla_x \vc{V} \, d\vc{x} = \int_{\Omega} \vc{g} \one \vc{V} \, d\vc{x}
\end{align}

and set

\begin{equation}
\sigma^{\rm eff}_{pq} {:=} \mean{\mathsf{K}_{pqij}} \frac{\partial u^0_i}{\partial x_j}, \quad \mu^{\rm eff}_{pqr} {:=} (\mathsf{A}^{\rm eff})^{pqr}_{ijk} \frac{\partial^2 u^0_i}{\partial x_j \partial x_k}. 
\end{equation}

Then \eqref{hs2:weak} becomes,

\begin{align}
&\int_{\Omega} \sigma^{\rm eff}_{pq} \frac{\partial V_p}{\partial x_q} \, d\vc{x} + \int_{\Omega} \mu^{\rm eff}_{pqr} \frac{\partial^2 V_p}{\partial x_r \partial x_q} \, d\vc{x} = \int_{\Omega} g_p V_p \, d\vc{x}.
\end{align}

Integrating by parts the first term once and the second term twice, we obtain,

\begin{align}
\int_{\Sigma} (\sigma^{\rm eff}_{pq} - \partial_{x_r} \mu^{\rm eff}_{pqr}) n_q  V_p \, ds - &\int_{\Omega} \partial_{x_q} (\sigma^{\rm eff}_{pq} - \partial_{x_r}\mu^{\rm eff}_{pqr}) V_p \, d\vc{x} \nonumber \\
+&\int_{\Sigma} \mu^{\rm eff}_{pqr} n_r \partial_{x_q} V_p \, ds= \int_{\Omega} g_p V_p \, d\vc{x}.
\end{align}

As before,  we decompose \AM{the boundary term} into normal and tangential components via,

\begin{equation}
\int_{\Sigma} \mu^{\rm eff}_{pqr} n_r \partial_{x_q} V_p \, ds = \int_{\Sigma} \mu^{\rm eff}_{pqr} n_q n_r n_l \partial_{x_l} V_p \, ds + \int_{\Sigma} \mu^{\rm eff}_{pqr} n_r \Pi_{lq} \partial_{x_l} V_p \, ds. 
\end{equation}

The first component of the above formula is a normal double traction while the second term we integrate by parts (on the surface $\Sigma$) using \eqref{eq:surface_parts} and obtain,

\begin{equation}
\int_{\Sigma} \mu^{\rm eff}_{pqr} n_r \Pi_{lq} \partial_{x_l} V_p \, ds = -\int_{\Sigma} \Pi_{ml} \partial_{x_l}(\mu^{\rm eff}_{pqr} n_r \Pi_{mq}) V_p \, ds 
+ \int_{\partial \Sigma} \jump{\mu^{\rm eff}_{pqr} n_r \nu_p}  V_p \, d\ell. 
\end{equation}

Thus, putting everything together, we have that \eqref{hs2:weak} is equivalent to the following identity:

\begin{gather}
\begin{aligned}
\int_{\Sigma} ((\sigma^{\rm eff}_{pq} - &\partial_{x_r} \mu^{\rm eff}_{pqr}) n_q - \Pi_{ml} \partial_{x_l}(\mu^{\rm eff}_{pqr} n_r \Pi_{mq}) )V_p \, ds - \int_{\Omega} \partial_{x_q} (\sigma^{\rm eff}_{pq} - \partial_{x_r}\mu^{\rm eff}_{pqr}) V_p \, d\vc{x} \\
+&\int_{\Sigma} \mu^{\rm eff}_{pqr} n_q n_r n_l \partial_{x_l} V_p \, ds + \int_{\partial \Sigma} \jump{\mu^{\rm eff}_{pqr} n_r \nu_p}  V_p \, d\ell = \int_{\Omega} g_p V_p \, d\vc{x}.
\end{aligned}
\end{gather}

From the above equation, we can recover the following boundary conditions on $\Sigma$ and $\partial \Sigma$,

\begin{itemize}
\item[-] Surface traction: $(\sigma^{\rm eff}_{pq} - \partial_{x_r} \mu^{\rm eff}_{pqr}) n_q - \Pi_{ml} \partial_{x_l}(\mu^{\rm eff}_{pqr} n_r \Pi_{mq})  = 0$ on $\Sigma_{\rm 1}$,
\item[-] A normal double traction: $\mu^{\rm eff}_{pqr} n_q n_r = 0$ on $\Sigma_{\rm 1}$,
\item[-] A line traction: $\jump{\mu^{\rm eff}_{pqr} n_r \nu_p}=0$ on $\partial \Sigma_{\rm 1}$,
\item[-] $\vc{u}^0 = \vc{0}$ and $\nabla \vc{u}^0 = \vc{0}$ on $\Sigma_{\rm 0}$ (the boundary conditions condition are a-priori in the function space),
\end{itemize}

which, jointly with the field equations, 

\begin{equation}
-\partial_{x_q} (\sigma^{\rm eff}_{pq} - \partial_{x_r}\mu^{\rm eff}_{pqr}) = g_p \text{ in } \mathcal{D}(\Omega),
\end{equation} 

build the complete set of equations governing equilibrium states for the second-gradient theory of reference \cite{ME68}, \cite{Germain73}.

\section*{Acknowledgements} 
The authors gratefully acknowledge the financial support by the Knowledge Foundation (project nr. KK 2020-0152). Moreover, we would like to express our gratitude to the anonymous reviewers for their many comments, suggestions, and corrections.
\small
\bibliography{ref2.bib}

\appendix \label{sec:appendix}
\numberwithin{equation}{section}
\section{Taylor expansion of the stored energy function around the equilibrium}

We perform a Taylor expansion of the stored energy function around the equilibrium. In principle we can continue this expansion and obtain any desired degree of accuracy of the nonlinear energy $W$. However, using the scaling introduce previously, we keep only the terms up to $\mathcal{O}(\alpha^3)$ leading to,

\begin{align*}
{\rm W} \left ( \vc{x}, \mathbb{F}, \mathbb{G},  \right ) 
= 
& {\rm {\rm W}} \left ( \vc{x}, \mathbb{I}, \mathbb{0} \right ) + 
\frac{ \partial {\rm W} }{ \partial F_{ij}} \left ( \vc{x}, \mathbb{I}, \mathbb{0} \right ) (F_{ij} - \delta_{ij}) 
+ \frac{ \partial {\rm W} }{ \partial G_{ijk}} \left ( \vc{x}, \mathbb{I}, \mathbb{0} \right ) \partial_{x_k} F_{ij} \\
+ & \frac{1}{2}\frac{ \partial^2 {\rm W} }{ \partial F_{ij} \partial F_{k\ell}} \left ( \vc{x}, \mathbb{I}, \mathbb{0} \right ) (F_{ij} - \delta_{ij})(F_{k\ell} - \delta_{k\ell}) \\
+ & \frac{ \partial^2 {\rm W} }{ \partial F_{ij} \partial G_{k\ell m}} \left ( \vc{x}, \mathbb{I}, \mathbb{0} \right ) (F_{ij} - \delta_{ij}) \partial_{x_m} F_{k\ell} \\
+ & \frac{1}{2} \frac{ \partial^2 {\rm W} }{ \partial G_{ijk} \partial G_{m\ell p}} \left ( \vc{x}, \mathbb{I}, \mathbb{0} \right ) \partial_{x_k} F_{ij} \partial_{x_p} F_{m\ell}
+ \mathcal{O}(\alpha^3).
\end{align*}

The potential energy at the equilibrium configuration is zero and, moreover, we assume that the material is stress free at the equilibrium configuration. Hence, the above expansion reduces to the following,

\begin{align*}
W \left ( \vc{x}, \mathbb{F}, \mathbb{G} \right ) 
= & 
\frac{1}{2}\frac{ \partial^2 W }{ \partial F_{ij} \partial F_{k\ell}} \left ( \vc{x}, \mathbb{I}, \mathbb{0} \right ) (F_{ij} - \delta_{ij})(F_{k\ell} - \delta_{k\ell}) \\
+ & 
\frac{ \partial^2 W }{ \partial F_{ij} \partial G_{k\ell m}} \left ( \vc{x}, \mathbb{I}, \mathbb{0} \right ) (F_{ij} - \delta_{ij}) \partial_{x_m} F_{k\ell} \\
+ & \frac{1}{2} \frac{ \partial^2 W }{ \partial G_{ijk} \partial G_{m\ell p}} \left ( \vc{x}, \mathbb{I}, \mathbb{0} \right ) \partial_{x_k} F_{ij} \partial_{x_p} F_{m\ell} + \mathcal{O}(\alpha^3).
\end{align*}

\subsection{Mechanical constitutive law for the stress and hyperstress up to $\mathcal{O}(\alpha^2)$}

The first constitutive law for the stress can be obtained from the above energy the following way,

\begin{equation*}
\sigma {=} \frac{\partial {\rm W}}{\partial \mathbb{F}} (\vc{x},\mathbb{F}, \mathbb{G}).
\end{equation*}

In components we have,

\begin{align*}
\sigma_{ij} 
& = \frac{ \partial^2 {\rm W} }{ \partial F_{ij} \partial F_{k\ell}} \left ( \vc{x}, \mathbb{I}, \mathbb{0} \right ) (F_{k\ell} - \delta_{k\ell}) + \frac{ \partial^2 W }{ \partial F_{ij} \partial G_{k\ell m}} \left ( \vc{x}, \mathbb{I}, \mathbb{0} \right ) \partial_{x_m} F_{k\ell} + \mathcal{O}(\alpha^{2}).
\end{align*}

Set,

\begin{align*}
\mathsf{K}_{ijk\ell} &{:=} \frac{ \partial^2 {\rm W} }{ \partial F_{ij} \partial F_{k\ell}} \left ( \vc{x}, \mathbb{I}, \mathbb{0} \right ), \quad
\mathsf{S}_{ij}^{k\ell m} {:=} \frac{ \partial^2 {\rm W} }{ \partial F_{ij} \partial G_{k \ell m}} \left ( \vc{x}, \mathbb{I}, \mathbb{0} \right ).
\end{align*}

In more compact form we can write,

\begin{equation}
\sigma_{ij} {=} \mathsf{K}_{ijk\ell} \frac{\partial u_k}{\partial x_\ell} + \mathsf{S}_{ij}^{k \ell m} \frac{\partial^2 u_k}{\partial x_m \partial x_\ell}.
\end{equation}

The constitutive law for the hyperstress can be obtained,

\begin{equation*}
\mu {=} \frac{\partial {\rm W}}{\partial \mathbb{G}} (\vc{x},\mathbb{F}, \mathbb{G}).
\end{equation*}

In components we obtain,

\begin{equation}
\mu_{ijk} = \frac{ \partial^2 {\rm W} }{ \partial F_{n\ell} \partial G_{ijk}} \left ( \vc{x}, \mathbb{I}, \mathbb{0} \right ) (F_{n\ell} - \delta_{n\ell}) + 
\frac{ \partial^2 {\rm W} }{ \partial G_{nk\ell} \partial G_{ijk}} \left ( \vc{x}, \mathbb{I}, \mathbb{0} \right ) \partial_{x_\ell}F_{nk} + \mathcal{O}(\alpha^{2}).
\end{equation}

If we set,

\begin{equation}
\mathsf{A}^{ijk}_{n\ell p} {:=} \frac{ \partial^2 {\rm W} }{ \partial G_{n\ell p} \partial G_{ijk}} \left ( \vc{x}, \mathbb{I}, \mathbb{0} \right ),
\end{equation}

then we can compactly write,

\begin{equation}
\mu_{ijk} {=} \mathsf{A}^{ijk}_{n\ell p} \frac{\partial^2 u_n}{\partial x_\ell \partial x_p} + \mathsf{S}_{n\ell}^{ijk} \frac{\partial u_n}{\partial x_\ell}.
\end{equation}

\end{document}